\documentclass[12pt,reqno]{amsart}

\usepackage[margin=1.15in]{geometry}
\usepackage{amsmath,amssymb,amsthm,mathrsfs}
\usepackage{tikz}
\usepackage{tikz-cd}
\usetikzlibrary{decorations.markings,arrows.meta}
\usepackage{graphicx}
\usepackage{enumitem}
\usepackage[colorlinks=true,linkcolor=blue,citecolor=blue,urlcolor=blue]{hyperref}
\usepackage{cleveref}

\theoremstyle{plain}
\newtheorem{theorem}{Theorem}[section]
\newtheorem{proposition}[theorem]{Proposition}
\newtheorem{lemma}[theorem]{Lemma}
\newtheorem{corollary}[theorem]{Corollary}
\newtheorem{conjecture}[theorem]{Conjecture}

\newtheorem*{theoremA}{Theorem A}
\newtheorem*{theoremB}{Theorem B}
\newtheorem*{theoremC}{Theorem C}

\theoremstyle{definition}
\newtheorem{definition}[theorem]{Definition}
\newtheorem{example}[theorem]{Example}
\newtheorem{remark}[theorem]{Remark}
\newtheorem{question}{Question}
\newtheorem{convention}[theorem]{Convention}

\newcommand{\Z}{\mathbb{Z}}
\newcommand{\Q}{\mathbb{Q}}
\newcommand{\F}{\mathbb{F}}
\newcommand{\CP}{\mathbb{CP}}
\newcommand{\RP}{\mathbb{RP}}
\newcommand{\Sph}{\mathbb{S}}

\newcommand{\X}{\mathcal{X}}
\newcommand{\Xo}{\X_{o}}
\newcommand{\Xe}{\X_{e}}

\newcommand{\Kh}{\mathit{Kh}_{e}}
\newcommand{\Khr}{\widetilde{\mathit{Kh}}_{e}}
\newcommand{\Kho}{\mathit{Kh}_{o}}
\newcommand{\Khro}{\widetilde{\mathit{Kh}}_{o}}

\newcommand{\Sq}{\mathrm{Sq}}
\newcommand{\Sm}{S}
\newcommand{\wed}{\vee}

\newcommand{\rk}{\operatorname{rk}}

\DeclareMathOperator{\Tor}{Tor}
\DeclareMathOperator{\Ext}{Ext}
\DeclareMathOperator{\Hom}{Hom}
\DeclareMathOperator{\Cone}{Cone}

\begin{document}

\title[Skein reduction for odd Khovanov spectra]{Skein reduction and stabilization\\
       for the odd Khovanov spectra of $T(2,m)\sqcup T(2,n)$}

\author{Krishnendu Kar}
\address{Department of Mathematics, Louisiana State University,
         Baton Rouge, LA 70803, USA}
\email{kkar2@lsu.edu}

\date{\today}

\subjclass[2020]{57K18, 55P42, 55S10}
\keywords{Khovanov homology, odd Khovanov homology, stable homotopy type,
          Steenrod square, skein cofiber sequence, torus links}

\begin{abstract}
For split links the even Khovanov spectrum is a smash product of the spectra of the
pieces; the odd one is not. For $T(2,m)\sqcup T(2,n)$ we show that the odd spectrum is, in
each quantum grading, a wedge of spheres and copies of a suspension of $\CP^{2}$. Using the
skein cofiber sequence at the level of spectra, we prove that for fixed $q-m-n$ the number of
these copies is determined by finitely many small links and stabilizes for large $m$ and $n$. The $\CP^{2}$
summand of $T(2,3)\sqcup T(2,3)$, due to Sch\"utz and Rajapakse, persists in quantum grading
$m+n+8$ for all $m,n\ge3$. We conjecture that the number
of copies of $\CP^{2}$ equals the number of copies of $\RP^{2}\wedge\RP^{2}$ in the even
spectrum.
\end{abstract}

\maketitle

\setcounter{tocdepth}{1}
\tableofcontents

\section{Introduction}

\subsection{Motivation}
Lipshitz and Sarkar refined Khovanov homology to a stable homotopy type: for each quantum
grading $q$ a finite CW spectrum $\Xe^{q}(L)$ whose reduced cohomology is
$\Kh^{*,q}(L;\Z)$ \cite{LS14a}. The spectrum carries more information than its cohomology:
it distinguishes some links with isomorphic Khovanov homology \cite{Seed}, and Lipshitz and
Sarkar used its second Steenrod square to construct refinements of Rasmussen's slice genus
bound \cite{Rasmussen10,LS14b}. Using signed Burnside functors, Sarkar, Scaduto and
Stoffregen constructed the analogous spectrum $\Xo^{q}(L)$ \cite{SSS} for the odd Khovanov
homology of Ozsv\'ath, Rasmussen and Szab\'o \cite{ORS}, which agrees with Khovanov
homology modulo~$2$ but not over~$\Q$.

For a split link the even spectrum is determined by the spectra of the components:
Lawson, Lipshitz and Sarkar proved
\begin{equation}\label{eq:even-smash}
\Xe^{q}(L_1\sqcup L_2)\simeq
\bigvee_{q_1+q_2=q}\Xe^{q_1}(L_1)\wedge\Xe^{q_2}(L_2)
\end{equation}
\cite[\S9]{LLS}. Sarkar, Scaduto and Stoffregen asked whether the odd spectrum of a split
link can be expressed in terms of the spectra of its components \cite[\S5.7, (q-6)]{SSS}.
It is not the smash product: Sch\"utz's calculation for the split union of two positive
trefoils \cite[\S5.2]{Schutz2022}, together with Rajapakse's identification of the odd
Steenrod square \cite[Thm.~1.6]{Rajapakse}, gives
\begin{equation}\label{eq:schutz-example}
\Xo^{14}\bigl(T(2,3)\sqcup T(2,3)\bigr)
\simeq\Sigma^2\CP^2\wed(S^5)^{\wed4},
\end{equation}
whereas each odd spectrum of a positive trefoil is a wedge of spheres. So this example
already excludes an odd analogue of \eqref{eq:even-smash}.

Integral cohomology does not detect a summand $\Sigma^{k}\CP^{2}$: the spectra
$\Sigma^{k}\CP^{2}$ and $S^{k+2}\wed S^{k+4}$ have isomorphic cohomology, and they are
distinguished by $\Sq^{2}$, which detects the Hopf map $\eta$ attaching the top cell. In the
even theory, Lipshitz and Sarkar found no such summand among prime links with at most $11$
crossings and asked whether one ever occurs \cite[Question~5.2]{LS14c}; Lobb, Orson and
Sch\"utz conjectured that none does \cite[Conj.~2.1]{LOS}. The example
\eqref{eq:schutz-example} shows that the odd theory behaves differently, and raises the
question of how many such summands the odd spectra of split links contain.

We study the family $L_{m,n}=T(2,m)\sqcup T(2,n)$, with both torus links positively
oriented and $m,n\ge2$. Its odd homology is free and, in each quantum grading,
supported in three consecutive homological degrees. The standard classification
in this range expresses the spectrum as a wedge of spheres and suspensions of
$\CP^2$. The remaining invariant is the number of $\CP^2$ summands; since the odd spectra
of the factors are wedges of spheres (\Cref{lem:factorspectra}), these summands occur in the split union although they are absent from the factors. Our main result reduces this number, for fixed $q-m-n$, to a
finite set of links. It also gives stabilization as either index grows and shows that the
summand in \eqref{eq:schutz-example} persists for all $m,n\ge3$. Thus the odd analogue of
\cite[Conj.~2.1]{LOS} fails for every $L_{m,n}$ with $m,n\ge3$; \Cref{ss:earlier}
discusses the consequences for the questions of \cite[\S5.7]{SSS}.

The reduction comes from the skein cofiber sequence. Its connecting map is zero
on integral cohomology, but need not be null-homotopic. We determine the part of
this map that can survive: a Hopf-map component from a single sphere in the
source. When that sphere is absent, the skein step splits off only spheres;
otherwise it can introduce at most one additional $\CP^2$ summand.
This gives both the reduction and a monotonicity statement for the number of
summands (\Cref{prop:dichotomy}).

\subsection{Statement of results}
Put
\[
\begin{gathered}
 d=\frac{q-m-n}{2},\qquad S_k=\{0\}\cup\{2,\ldots,k\},\\
 P(t)=P_{m,n}(t)=\#\{(i_1,i_2)\in S_m\times S_n:i_1+i_2=t\}.
\end{gathered}
\]
\begin{remark}\label{rem:parity}
By \Cref{lem:torusknot} and \Cref{cor:kunneth-torus}, $\Kho^{*,q}(L_{m,n};\Z)$ vanishes
unless $q\equiv m+n\pmod2$, and the spectrum is contractible in the remaining
gradings. We restrict to this parity, so $d$ is an integer.
\end{remark}

The homology ranks in degrees $d,d+1,d+2$ are $P(d),2P(d+1),P(d+2)$; we call $d$ the
\emph{window index} of $L_{m,n}$ at $q$.
Consequently there is a unique integer $\nu=\nu(m,n;q)\ge0$ such that
\begin{equation}\label{eq:intro-classification}
\begin{split}
\Xo^q(L_{m,n})\simeq{}& (\Sigma^{d-2}\CP^2)^{\wed\nu}
 \wed(S^d)^{\wed(P(d)-\nu)}\wed(S^{d+1})^{\wed 2P(d+1)}\wed(S^{d+2})^{\wed(P(d+2)-\nu)}.
\end{split}
\end{equation}
Here $\nu$ is the rank of $\Sq^2_1$ from homological degree $d$ to $d+2$.
The homology calculation and classification are given in
\Cref{sec:window,sec:classification}.

\begin{theoremA}
For every $m,n\ge2$, one has $\nu(m,n;q)=0$ when
\[
q\in\{m+n-4,\ m+n-2,\ m+n+2,\ 3(m+n)-2,\ 3(m+n)\}.
\]
Thus the spectrum is a wedge of spheres in these gradings.
\end{theoremA}
These are the values $d\in\{-2,-1,1,m+n-1,m+n\}$.
The interior value $d=1$ comes from the missing homological degree $1$ in
$\Khro(T(2,k))$.

For the remaining gradings define
\[
M(q)=\begin{cases}m,&q\ge3m+n,\\d,&q=3m+n-2,\\d+2,&q\le3m+n-4,\end{cases}
\qquad
N(q)=\begin{cases}n,&q\ge m+3n,\\d,&q=m+3n-2,\\d+2,&q\le m+3n-4,\end{cases}
\]
and $q^*=M(q)+N(q)+2d$.
\begin{theoremB}
If $0\le d\le m+n-2$ and $d\ne1$, then
\[
\Xo^q(L_{m,n})\simeq\Xo^{q^*}(L_{M(q),N(q)})\wed W,
\]
where $W$ is a wedge of spheres. In particular,
$\nu(m,n;q)=\nu(M(q),N(q);q^*)$.
\end{theoremB}
The target is a fixed point of this reduction. The finite set of targets has
an explicit description.
\begin{corollary}\label{cor:intro-irreducible}
For $d\ge0$, $d\ne1$, set
\[
I_d=\{\{a,b\}:a,b\in\{2,\ldots,d\}\cup\{d+2\},\ a+b\ge d+2\},
\]
where unordered pairs allow equal entries. Every value
$\nu(m,n;m+n+2d)$ with $m+n\ge d+2$ equals $\nu(a,b;a+b+2d)$ for some $\{a,b\}\in I_d$.
Moreover, $|I_0|=1$ and $|I_d|=d+\lfloor d^2/4\rfloor$ for $d\ge2$.
For $d<0$, $d=1$, or $d>m+n-2$, the value of $\nu(m,n;m+n+2d)$ is zero.
\end{corollary}
The full distinction between vanishing, reducible and terminal gradings is
recorded in \Cref{cor:trichotomy}. Here ``terminal'' means only that the
reduction makes no further step; it does not assert that the spectrum is
indecomposable.

For $m,n\ge d+2$ the target is $L_{d+2,d+2}$ and the homology ranks have stabilized.
\begin{theoremC}
Fix $d\ge3$ and put $\nu_d=\nu(d+2,d+2;4d+4)$. For all $m,n\ge d+2$,
\[
\begin{split}
\Xo^{m+n+2d}(L_{m,n})\simeq{}& (\Sigma^{d-2}\CP^2)^{\wed\nu_d}
 \wed(S^d)^{\wed(d-1-\nu_d)}\\
 &{}\wed(S^{d+1})^{\wed 2d}\wed(S^{d+2})^{\wed(d+1-\nu_d)}.
\end{split}
\]
\end{theoremC}
The skein argument also proves
$\nu(m,n;m+n+8)\ge1$ for all $m,n\ge3$, using the published trefoil calculation
\eqref{eq:schutz-example}. Determining whether this number is exactly one
reduces to the single rank $\nu(6,6;20)$. The numerical evidence and the
consequences of that rank computation are separated from the structural
results in \Cref{sec:examples}. Theorems~\ref{thm:vanishing}, \ref{thm:reduction}
and~\ref{thm:stabilization} also hold for the mirrors $T(2,-m)\sqcup T(2,-n)$, with an
invariant $\overline{\nu}$ in place of $\nu$ (\Cref{thm:mirror}).

\subsection{Relation to earlier work}
Our use of $\Sq^2_1$ relies on Rajapakse's identification with the Steenrod square
of $\Xo$.
The trefoil example and the failure of the odd smash-product formula are
already known \cite[\S5.2]{Schutz2022}, \cite[Thm.~1.3]{Rajapakse}.
The contribution here is the reduction across the two-parameter family,
including steps with a noncontractible source, and the resulting propagation
of this example.

The even skein cofibration appears in \cite{LS14a}, and flow-category
reductions are developed in \cite{LOS,SchutzMorse}. Willis proved stabilization
of the even Khovanov homotopy types of torus links \cite{Willis}.
Our stabilization concerns odd spectra of split unions; its proof uses the
explicit source of the skein sequence, without a smash-product formula.
The three-degree classification used here is a standard special case of
Chang-type classification \cite{Chang,LS14c}; we include a proof to make
its application to the connecting map explicit.

Although $L_{m,n}$ is alternating, it is split and hence not quasi-alternating:
its determinant is zero, whereas quasi-alternating links have positive
determinant \cite{OSz,MO}. Thus the usual quasi-alternating conclusions do not
exclude the $\CP^2$ summands found here. In \Cref{sec:questions} we compare
their number with the number of $\RP^2\wedge\RP^2$ summands in the even spectrum
and formulate the equality as a conjecture.

\subsection{Organization}
\Cref{sec:background,sec:window,sec:classification} give the background,
homology calculation and classification. The skein argument and reduction
occupy \Cref{sec:skein,sec:reduction}, followed by stabilization in
\Cref{sec:stabilization}. \Cref{sec:examples,sec:questions} contain examples,
computational evidence and open questions, and \Cref{sec:mirror} treats the mirror family.

\subsection*{Acknowledgments}
The author thanks Michael Willis and Martin Frankland for many helpful discussions.

\section{Background and conventions}\label{sec:background}

\subsection{Odd Khovanov homology}\label{ss:oddkh}

We work throughout with unreduced odd Khovanov homology $\Kho^{i,q}(L;\Z)$ of
Ozsv\'ath--Rasmussen--Szab\'o \cite{ORS}, in the conventions of \cite{SSS}. Let $D$ be an
oriented link diagram with $n$ ordered crossings, $n_{+}=n_{+}(D)$ of them positive and
$n_{-}=n_{-}(D)$ negative. The cube $\underline{2}^{\,n}$ is the category whose objects are the vertices
$v\in\{0,1\}^{n}$, with a unique morphism $\phi_{u,v}\colon u\to v$ whenever $u\ge v$
coordinatewise; we write $|v|=v_{1}+\dots+v_{n}$, and call $\phi_{u,v}$ an \emph{edge} if
$|u|=|v|+1$. For a vertex $v$ let $D_{v}$ be the resolution of $D$ with the
$v_{i}$-resolution at the $i$-th crossing (\Cref{fig:resolutions}), and $Z(D_{v})$ its set of
circles, ordered in some way.

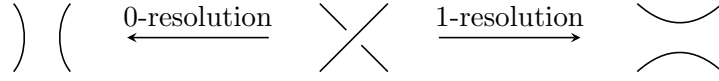
\begin{figure}[ht]
\centering
\begin{tikzpicture}[scale=0.75, line width=0.6pt, >=stealth]
\draw (-0.6,0.6) -- (0.6,-0.6);
\draw[white, line width=8pt] (-0.2,-0.2) -- (0.2,0.2);
\draw (-0.6,-0.6) -- (0.6,0.6);
\draw[->] (-1.5,0) -- (-4,0) node[midway, above] {\small $0$-resolution};
\begin{scope}[shift={(-9,0)}]
\draw (3.0,0.6) .. controls (3.25,0.3) and (3.25,-0.3) .. (3.0,-0.6);
\draw (4.0,0.6) .. controls (3.75,0.3) and (3.75,-0.3) .. (4.0,-0.6);
\end{scope}
\draw[->] (1.5,0) -- (4,0) node[midway, above] {\small $1$-resolution};
\begin{scope}[shift={(9,0)}]
\draw (-4.0,0.6) .. controls (-3.5,0.15) and (-3.0,0.15) .. (-2.5,0.6);
\draw (-4.0,-0.6) .. controls (-3.5,-0.15) and (-3.0,-0.15) .. (-2.5,-0.6);
\end{scope}
\end{tikzpicture}
\caption{A crossing (middle), its $0$-resolution (left) and its $1$-resolution (right).}
\label{fig:resolutions}
\end{figure}

The odd Khovanov complex $\mathit{CKh}_{o}(D)$ is the direct sum over the vertices of the
exterior algebras $\Lambda(Z(D_{v}))$ over $\Z$, with the bases of \emph{Khovanov generators}
\[
\mathrm{Kg}(v)=\{a_{1}\wedge\cdots\wedge a_{k}\ :\ a_{i}\in Z(D_{v}),\ a_{1}>\cdots>a_{k}\};
\]
its differential is a signed sum of \emph{edge maps} $\Lambda(Z(D_{v}))\to\Lambda(Z(D_{u}))$, one for each
edge $\phi_{u,v}$, induced by the merge or split of circles from $D_{v}$ to $D_{u}$
\cite{ORS}, \cite[\S2.3]{SSS}. A generator $x=a_{1}\wedge\cdots\wedge a_{k}\in\mathrm{Kg}(v)$
has bigrading
\[
i(x)=|v|-n_{-},\qquad q(x)=|Z(D_{v})|-2k+|v|+n_{+}-2n_{-}
\]
\cite[Def.~2.6]{SSS}; the differential raises $i$ by one and preserves $q$, and
$\Kho^{i,q}(L;\Z)$ is the cohomology of $\mathit{CKh}_{o}(D)$ in bidegree $(i,q)$.

We write $\{n\}$ for the shift with $(M\{n\})^{i,q}=M^{i,q-n}$, and put $\delta:=q-2i$. We
recall two results from \cite{ORS}.

\begin{proposition}[{\cite[Proposition 1.8]{ORS}}]\label{prop:orsplit}
For every link $L$ there is a bigraded abelian group $\Khro(L;\Z)$, the reduced odd
Khovanov homology, with
\[
  \Kho(L;\Z)\;\cong\;\Khro(L;\Z)\{-1\}\ \oplus\ \Khro(L;\Z)\{1\}.
\]
\end{proposition}

\begin{proposition}[{\cite[Definition 5.1, Proposition 5.2 and the remark following it]{ORS}}]\label{prop:orqa}
If $L$ is quasi-alternating, then $\Khro(L;\Z)$ is free of rank $\det(L)$ and supported
on the $\delta$-diagonal $\delta=\sigma(L)$, where the signature $\sigma$ is normalized so
that positive links have positive signature.
\end{proposition}

Two things distinguish Proposition~\ref{prop:orsplit} from its even analogue. First,
$\Khro$ requires no choice of basepoint: reduced Khovanov homology depends on a chosen
component of $L$, and different choices can give different answers, whereas $\Khro$ is a
link invariant. Second, the even statement is not a splitting at all, but a long exact
sequence relating $\Kh$ and $\Khr$ whose connecting map is in general non-zero; the odd
splitting holds over $\Z$.

Now we compute odd Khovanov homology for $T(2,k)$ using these propositions.
\begin{lemma}\label{lem:torusknot}
For $k\ge1$,
\[
\Kho^{\,i,q}\bigl(T(2,k);\Z\bigr)\ \cong\ \begin{cases}\Z, & i\in S_k\ \text{ and }\ q\in\{k+2i-2,\,k+2i\},\\ 0,&\text{otherwise.}\end{cases}
\]
In particular the total rank is $2k$, and no homological degree $1$ class occurs.
\end{lemma}

\begin{proof}

$T(2,k)$ is alternating and non-split, hence quasi-alternating, so \Cref{prop:orqa}
applies: $\Khro\bigl(T(2,k);\Z\bigr)$ is free of rank $\det T(2,k)=k$ and supported on
the diagonal $\delta=\sigma\bigl(T(2,k)\bigr)=k-1$, that is, in bidegrees $(i,2i+k-1)$.

Being free and concentrated on one diagonal, it is determined by its graded Euler
characteristic, which is the normalized Jones polynomial:
\[
\sum_{i}(-1)^{i}\,\rk\Khro^{\,i,\,2i+k-1}\bigl(T(2,k)\bigr)\,q^{\,2i+k-1}\ =\ J_{T(2,k)}(q).
\]
No cancellation can occur between distinct $i$, since distinct $i$ give distinct powers
of $q$; so the ranks are exactly the absolute values of the coefficients of
$J_{T(2,k)}$. An induction on $k$ using the Kauffman bracket skein relation gives
\begin{equation}\label{eq:jonestorus}
J_{T(2,k)}(q)\ =\ q^{\,k-1}\ +\ \sum_{i=2}^{k}(-1)^{i}q^{\,2i+k-1}\ =\ \sum_{i\in S_k}(-1)^{i}q^{\,2i+k-1},
\end{equation}
a polynomial with $k$ non-zero coefficients, all equal to $\pm1$. Hence
$\Khro^{\,i,2i+k-1}(T(2,k);\Z)\cong\Z$ for $i\in S_k$ and vanishes otherwise. The
homological degrees are therefore exactly $S_k=\{0\}\cup\{2,\dots,k\}$; the absent
degree $1$ is the absent monomial $q^{\,k+1}$ in \eqref{eq:jonestorus}, the first term of
the alternating tail starting instead at $i=2$.

Finally, by \Cref{prop:orsplit} the unreduced group is
$\Khro\{-1\}\oplus\Khro\{1\}$, so each generator in bidegree $(i,2i+k-1)$ contributes a
$\Z$ in bidegrees $(i,2i+k-2)$ and $(i,2i+k)$. This is the stated answer, of total rank
$2k$.
\end{proof}

For split unions the algebraic K\"unneth theorem gives the following.
\begin{lemma}\label{lem:kunneth}
Let $L_1$ and $L_2$ be links.
\begin{enumerate}[label=\textup{(\roman*)}]
\item For all $i,q$ there is a short exact sequence
\begin{equation}\label{eq:kunneth-ses}
\begin{split}
0&\longrightarrow\bigoplus_{\substack{i_1+i_2=i\\ q_1+q_2=q}}
   \Kho^{i_1,q_1}(L_1;\Z)\otimes_{\Z}\Kho^{i_2,q_2}(L_2;\Z)
   \longrightarrow\Kho^{i,q}(L_1\sqcup L_2;\Z)\\[2pt]
 &\longrightarrow\bigoplus_{\substack{i_1+i_2=i+1\\ q_1+q_2=q}}
   \Tor^{\Z}_{1}\bigl(\Kho^{i_1,q_1}(L_1;\Z),\Kho^{i_2,q_2}(L_2;\Z)\bigr)
   \longrightarrow0.
\end{split}
\end{equation}
\item If at least one of $L_1,L_2$ is quasi-alternating, the $\Tor$ term vanishes and
\eqref{eq:kunneth-ses} reduces to an isomorphism of bigraded groups
\begin{equation}\label{eq:kunneth}
\Kho\bigl(L_1\sqcup L_2;\Z\bigr)\ \cong\ \Kho(L_1;\Z)\otimes_{\Z}\Kho(L_2;\Z).
\end{equation}
\end{enumerate}
\end{lemma}

\begin{proof}
(i) For diagrams $D_1,D_2$ in disjoint discs, Putyra's disjoint-union formula
\cite[Theorem~6.7]{PutyraTriply}, in its integral odd specialization
\cite[Theorem~2.10]{DLS}, gives an isomorphism of bigraded complexes
\begin{equation}\label{eq:kunneth-chain}
\mathit{CKh}_o(D_1\sqcup D_2)
\cong\mathit{CKh}_o(D_1)\otimes_{\Z}\mathit{CKh}_o(D_2).
\end{equation}
The bigradings match without shift because
$n_{\pm}(D_1\sqcup D_2)=n_{\pm}(D_1)+n_{\pm}(D_2)$.
Apply the algebraic K\"unneth theorem for free complexes
\cite[Theorem~3.6.1]{Weibel}, separately in each quantum grading.
The $\Tor$ term has $i_1+i_2=i+1$ because our differential is
cohomological.

(ii) Say $L_1$ is quasi-alternating. By \Cref{prop:orqa} the group $\Khro(L_1;\Z)$ is
free, so by \Cref{prop:orsplit}
$\Kho(L_1;\Z)\cong\Khro(L_1;\Z)\{-1\}\oplus\Khro(L_1;\Z)\{1\}$ is a direct sum of two
free abelian groups, hence free. Hence the $\Tor$ term of \eqref{eq:kunneth-ses} vanishes and the left-hand map is an isomorphism. Assembling
over all $i,q$ gives \eqref{eq:kunneth}.
\end{proof}

Since $T(2,m)$ is quasi-alternating (see the proof of \Cref{lem:torusknot}),
\Cref{lem:kunneth}(ii) applies to the links of this paper.
\begin{corollary}\label{cor:kunneth-torus}
For all $m,n\ge1$,
\[
\Kho\bigl(T(2,m)\sqcup T(2,n);\Z\bigr)\ \cong\
\Kho\bigl(T(2,m);\Z\bigr)\otimes_{\Z}\Kho\bigl(T(2,n);\Z\bigr),
\]
a free abelian group.
\end{corollary}

\subsection{The odd stable homotopy type}\label{ss:oddspectrum}

Sarkar, Scaduto and Stoffregen \cite{SSS} construct, for each quantum grading $q$, a spectrum
$\Xo^{q}(L)$ with $\widetilde H^{\,i}(\Xo^{q}(L);\Z)\cong\Kho^{\,i,q}(L;\Z)$, adapting the
Burnside category approach of \cite{LLS}. We recall what is used in \Cref{sec:skein}.

A \emph{signed correspondence} from a finite set $X$ to a finite set $Y$ is a finite set $A$
with maps $s\colon A\to X$, $t\colon A\to Y$ and $\sigma\colon A\to\{\pm1\}$, the source,
target and sign. Correspondences compose by fiber product,
$B\circ A=\{(b,a) : t(a)=s(b)\}$, with the signs multiplied, and finite sets, signed
correspondences and the bijections between them form the signed Burnside category
$\mathscr B_{\sigma}$ \cite[\S3.2]{SSS}. A \emph{signed Burnside functor}
$F\colon\underline{2}^{\,n}\to\mathscr B_{\sigma}$ assigns a finite set $F(v)$ to each vertex
and a signed correspondence $F(\phi_{u,v})$ from $F(u)$ to $F(v)$ to each $u\ge v$, compatibly
with composition up to coherent bijections \cite[\S3.5]{SSS}. The \emph{odd Khovanov Burnside
functor} $F_{o}=F_{o}(D)$ has $F_{o}(v)=\mathrm{Kg}(v)$ and, on an edge $\phi_{u,v}$,
\[
F_{o}(\phi_{u,v})=\{(y,x)\in\mathrm{Kg}(v)\times\mathrm{Kg}(u)\ :\ \varepsilon_{x,y}\ne0\},
\]
with $s(y,x)=x$, $t(y,x)=y$ and $\sigma(y,x)=\varepsilon_{x,y}$, where
$\varepsilon_{x,y}\in\{-1,0,1\}$ is the coefficient of $x$ in the image of $y$ under the edge
map of $\mathit{CKh}_{o}(D)$ \cite[\S5.1]{SSS}. It splits as $F_{o}=\bigsqcup_{q}F_{o}^{q}$ by
quantum grading.

\emph{Realization.} A signed correspondence $A$ from $X$ to $Y$ is realized by a \emph{box
map} from $\bigvee_{X}S^{k}$ to $\bigvee_{Y}S^{k}$: the sphere of $x\in X$ is collapsed onto disjoint
sub-boxes, one for each $a\in s^{-1}(x)$, and the sub-box of $a$ is mapped onto the sphere of
$t(a)$ by the identity if $\sigma(a)=+1$ and by the reflection in the first coordinate if
$\sigma(a)=-1$. On $\widetilde H_{k}$ it sends $x$ to $\sum_{a\in s^{-1}(x)}\sigma(a)\,t(a)$
\cite[\S4.1]{SSS}. For $k$ large, box maps realizing the correspondences of $F$ assemble into a
homotopy coherent cube of wedges of $k$-spheres, unique up to equivalence; its iterated
mapping cone $\|F\|_{k}$ is a CW complex with one $(|v|+k)$-cell for each $x\in F(v)$
\cite[\S4]{SSS}. For $r\in\Z$ one puts $|\Sigma^{r}F|=\Sigma^{r-k}\Sigma^{\infty}\|F\|_{k}$.

\begin{definition}[{\cite[Def.~5.2]{SSS}}]\label{def:oddspectrum}
The odd Khovanov spectrum of $D$ in quantum grading $q$ is
$\Xo^{q}(D)=\bigl|\Sigma^{-n_{-}}F_{o}^{q}(D)\bigr|$.
\end{definition}

By \cite[Thm.~1.1]{SSS} its homotopy type is a link invariant, and its reduced cellular
cochain complex is $\mathit{CKh}_{o}^{*,q}(D)$, with the cells corresponding to the Khovanov
generators.

\begin{definition}[{\cite[\S3.7, Def.~3.4]{SSS}}]\label{def:cofseq}
A \emph{sub-functor} of $F\colon\underline{2}^{\,n}\to\mathscr B_{\sigma}$ is a signed
Burnside functor $G\colon\underline{2}^{\,n}\to\mathscr B_{\sigma}$ such that
\begin{enumerate}[label=\textup{(\roman*)}]
\item $G(v)\subset F(v)$ for every $v$;
\item $G(\phi_{u,v})\subset F(\phi_{u,v})$ for every $u\ge v$, with source, target and sign
restricted from $F(\phi_{u,v})$;
\item $s^{-1}(x)\subset G(\phi_{u,v})$ for every $u\ge v$ and every $x\in G(u)$, where
$s\colon F(\phi_{u,v})\to F(u)$ is the source map of $F$.
\end{enumerate}
Then $H(v)=F(v)\smallsetminus G(v)$ and $H(\phi_{u,v})=t^{-1}\bigl(H(v)\bigr)\subset
F(\phi_{u,v})$ define a signed Burnside functor $H$, the \emph{quotient} of $F$ by $G$, and
$G\to F\to H$ is called a \emph{cofibration sequence} of signed Burnside functors.
\end{definition}

\begin{lemma}\label{lem:realization}
Let $r\in\Z$.
\begin{enumerate}[label=\textup{(\roman*)}]
\item A cofibration sequence $G\to F\to H$ of signed Burnside functors realizes to a
cofibration sequence of spectra $|\Sigma^{r}G|\to|\Sigma^{r}F|\to|\Sigma^{r}H|$
\cite[Lem.~4.15]{SSS}.
\item For the disjoint union $F_{1}\sqcup F_{2}$ of signed Burnside functors, formed
vertexwise and edgewise,
$|\Sigma^{r}(F_{1}\sqcup F_{2})|\simeq|\Sigma^{r}F_{1}|\wed|\Sigma^{r}F_{2}|$
\cite[Lem.~4.25]{SSS}.
\end{enumerate}
\end{lemma}

Sch\"utz \cite{Schutz2022} constructs two second Steenrod squares $\Sq^{2}_{\varepsilon}$,
$\varepsilon\in\Z/2$, on $\Kho(L;\F_2)$ combinatorially; he does not show that they are
induced by a spectrum. That identification is due to Rajapakse \cite[Thm.~1.6]{Rajapakse}:
writing $\mathcal{X}_{l}(L)$, $l\ge0$, for the spectra obtained by the construction above when
a sign $-1$ is realized by the reflection in the first $l$ coordinates rather than in the
first one \cite[Rem.~4.26 and Rem.~5.6]{SSS}, so that $\mathcal{X}_{1}(L)=\Xo(L)$,
\[
\Sq^{2}\big|_{\mathcal{X}_{1}(L)}=\Sq^{2}_{1},\qquad
\Sq^{2}\big|_{\mathcal{X}_{3}(L)}=\Sq^{2}_{0},
\]
as operations on odd Khovanov homology with $\F_2$ coefficients. It is this theorem that
lets us read Sch\"utz's computation \eqref{eq:schutz-example} as a statement about a
homotopy type, and every statement in this paper asserting $\nu\ge1$ rests on it.
\begin{convention}
All spectra are finite CW spectra; $\wed$ denotes wedge, $\Sph^{n}$ the sphere spectrum in degree $n$, $\Sm^{n}$ the $n$-sphere, and $\CP^{2}$ the complex projective plane, so that $\Sigma^{d-2}\CP^{2}$ has cells in dimensions $d$ and $d+2$ and $\widetilde H^{d+2}=\widetilde H^{d}=\Z$. In the Introduction and from \Cref{cor:classification} on, $\Sm^{n}$ and $\CP^{2}$ also denote the corresponding suspension spectra; see the paragraph after \Cref{cor:classification}. 
\end{convention}

\subsection{Chang complexes and $\CP^{2}$}\label{ss:chang}

The two ways $\Sq^{2}$ can be non-zero in this setting, shown in \Cref{fig:chang}, are worth
contrasting explicitly, since the distinction drives the whole paper.

\begin{figure}[ht]
\centering
\begin{tikzpicture}[x=1cm,y=1.05cm]
\foreach \y/\lab in {0/d,1/d{+}1,2/d{+}2} {
  \draw[black!20] (-0.5,\y) -- (5.8,\y);
  \node[anchor=east,font=\small] at (-0.6,\y) {$\lab$};
}
\draw[thick] (0.6,0) to[bend right=42] node[right] {$\eta$} (0.6,2);
\fill (0.6,0) circle (2.4pt);
\fill (0.6,2) circle (2.4pt);
\node at (0.75,-0.75) {$\Sigma^{d-2}\CP^{2}$};
\begin{scope}[xshift=3.1cm]
\draw[thick] (0,0) -- node[left] {$2$} (0,1);
\draw[thick] (2,1) -- node[right] {$2$} (2,2);
\draw[thick] (0,0) to[bend left=14] node[above left,pos=0.5] {$\eta$} (2,2);
\fill (0,0) circle (2.4pt);
\fill (0,1) circle (2.4pt);
\fill (2,1) circle (2.4pt);
\fill (2,2) circle (2.4pt);
\node at (1.0,-0.75) {$\Sigma^{d-2}(\RP^{2}\wedge\RP^{2})$};
\end{scope}
\end{tikzpicture}
\caption{Cell diagrams of $\Sigma^{d-2}\CP^{2}$ and of $\Sigma^{d-2}(\RP^{2}\wedge\RP^{2})$.
A dot is a cell, drawn at the height of its dimension; a straight edge marked $2$ is a
component of degree $2$ of an attaching map, and a curved edge marked $\eta$ is a component
equal to $\eta$. On the left, $\Sigma^{d-2}\CP^{2}=\Cone(\eta\colon\Sm^{d+1}\to\Sm^{d})$. On
the right, in a suitable basis of cells, one $(d+1)$-cell is attached to the $d$-cell by
$2$, the other is attached trivially, and the $(d+2)$-cell is attached by $2$ to the second
$(d+1)$-cell and by $\eta$ to the $d$-cell.}
\label{fig:chang}
\end{figure}

\begin{itemize}[leftmargin=*]
\item If the homology is \emph{free}, then $\Sq^{1}=0$, but $\Sq^{2}$ can still be non-zero: it
detects a $(d+2)$-cell attached to a $d$-cell by $\eta$, as in
$\Cone(\eta\colon \Sm^{d+1}\to \Sm^{d})=\Sigma^{d-2}\CP^{2}$, whose cohomology is $\Z$ in
degrees $d$ and $d+2$.
\item If the homology has $2$-torsion, a non-zero $\Sq^{2}$ does not force a summand
$\Sigma^{k}\CP^{2}$. For instance, the Chang complex \cite{Chang} $\RP^{2}\wedge\RP^{2}$ has
reduced integral cohomology $\Z/2$ in degrees $3$ and $4$, so it has no such summand,
whereas $\Sq^{2}(x\otimes x)=\Sq^{1}x\otimes\Sq^{1}x\neq0$ for the generator
$x\in H^{1}(\RP^{2};\F_2)$, by the Cartan formula.
\end{itemize}

Since $\Kho(L_{m,n};\Z)$ is free (\Cref{cor:kunneth-torus}), only the first possibility
occurs here.

\section{The window structure of a split union}\label{sec:window}
The odd Khovanov homology of $L_{m,n}$ is free and, in each quantum grading, supported
in three consecutive homological degrees. We compute the ranks and the spectra of the
individual factors.
\begin{lemma}\label{lem:freeduality}
Let $X$ be a finite spectrum. Then $\widetilde H^{\,*}(X;\Z)$ is free if and only if
$\widetilde H_{*}(X;\Z)$ is free, and in that case
$\widetilde H_{i}(X;\Z)\cong\widetilde H^{\,i}(X;\Z)$ for every $i$.
\end{lemma}

\begin{proof}
$\widetilde H_{*}(X;\Z)$ is finitely generated in each degree and vanishes in all but
finitely many, so the universal coefficient theorem gives
\[
\widetilde H^{\,i}(X;\Z)\ \cong\ \Hom\bigl(\widetilde H_{i}(X),\Z\bigr)\ \oplus\
\Ext\bigl(\widetilde H_{i-1}(X),\Z\bigr),
\]
and for finitely generated $A$ the group $\Ext(A,\Z)$ is the torsion subgroup of $A$. If
every $\widetilde H_{i}(X;\Z)$ is free, the $\Ext$ terms vanish and
$\widetilde H^{\,i}\cong\Hom(\widetilde H_{i},\Z)\cong\widetilde H_{i}$. Conversely, if
every $\widetilde H^{\,i}(X;\Z)$ is free, then each finite group
$\Ext\bigl(\widetilde H_{i-1}(X),\Z\bigr)$ is a direct summand of a free abelian group and
hence trivial, so every $\widetilde H_{i}(X;\Z)$ is torsion free, therefore free.
\end{proof}

\begin{convention}\label{conv:variance}
Every spectrum occurring below has free integral (co)homology, so by
\Cref{lem:freeduality} $\widetilde H_{i}(X;\Z)\cong\widetilde H^{\,i}(X;\Z)$ for all $i$:
the two vanish in the same degrees and have the same ranks. We therefore speak of
\emph{the degrees in which $X$ is concentrated} and of \emph{the rank of $X$ in degree
$i$} without specifying the variance. Statements about $\Kho$ and about
$\Sq^{2}_{\varepsilon}$ are cohomological, while cell structures, connectivity and
Hurewicz arguments are homological; \Cref{lem:freeduality} is what licenses passing
between the two, and we do so without further comment.
\end{convention}

Now we record the spectra of the factors.

\begin{lemma}\label{lem:factorspectra}
Let $k\ge1$ and $q\equiv k\pmod 2$, and put $t:=\tfrac{q-k}{2}$, $a:=[\,t\in S_k\,]$,
$b:=[\,t+1\in S_k\,]$, where $[\,\cdot\,]$ is $1$ if the condition holds and $0$
otherwise. Then
\[
\Xo^{q}\bigl(T(2,k)\bigr)\ \simeq\ (\Sph^{t})^{\wed a}\ \wed\ (\Sph^{t+1})^{\wed b}.
\]
Equivalently, for every $Q$ the spectrum $\Xo^{Q}\bigl(T(2,k)\bigr)$ is a wedge of spheres
carrying one sphere $\Sph^{\,i}$ for each $i\in S_k$ with $Q\in\{k+2i-2,\,k+2i\}$, and no
other cells.
\end{lemma}
\begin{proof}
By \Cref{lem:torusknot} and \Cref{conv:variance}, $X:=\Xo^{q}(T(2,k))$ has free homology of
ranks $a$ and $b$ in degrees $t$ and $t+1$, and none in other degrees. For $N$ large there is
a simply connected finite CW complex $Y$ with $\Sigma^{\infty}Y\simeq\Sigma^{N}X$; put
$u:=t+N$ and take $N$ so large that $u\ge2$. By \cite[Prop.~4C.1]{Hatcher}, $Y$ is homotopy
equivalent to a complex with one $0$-cell, $a$ cells of dimension $u$ and $b$ cells of
dimension $u+1$. Its cellular chain groups have the ranks of its homology, so the cellular
differential vanishes. The $u$-skeleton is $(\Sm^{u})^{\wed a}$, and the cellular boundary of
a $(u+1)$-cell is the image of its attaching map under the Hurewicz isomorphism
$\pi_{u}\bigl((\Sm^{u})^{\wed a}\bigr)\cong\widetilde H_{u}\bigl((\Sm^{u})^{\wed a};\Z\bigr)$.
So every attaching map is null, $Y\simeq(\Sm^{u})^{\wed a}\wed(\Sm^{u+1})^{\wed b}$, and
desuspending $N$ times gives the claim.
\end{proof}
For the split union, \Cref{lem:torusknot} and \Cref{cor:kunneth-torus} determine the odd
Khovanov homology of $L_{m,n}$ in each quantum grading.

\begin{lemma}\label{lem:window}
Let $m,n\ge1$ and $q\equiv m+n\pmod 2$. Then $\Kho^{\,i,q}(L_{m,n};\Z)$ is
free, zero unless $i\in\{d,d+1,d+2\}$, and
\[
\rk\Kho^{\,d+\varepsilon,\,q}\bigl(L_{m,n};\Z\bigr)\ =\ \binom{2}{\varepsilon}\,P(d+\varepsilon),
\qquad \varepsilon=0,1,2 .
\]
Equivalently, $\Xo^{q}(L_{m,n})$ is concentrated in degrees $d,d+1,d+2$ with these ranks,
in the sense of \Cref{conv:variance}.
\end{lemma}

\begin{proof}
By \Cref{lem:torusknot}, $\Kho(T(2,k);\Z)$ has, for each $i\in S_k$, exactly two basis
elements: a \emph{lower} one in quantum grading $k+2i-2$ and an \emph{upper} one in
$k+2i$. By \Cref{cor:kunneth-torus} a basis of $\Kho(L_{m,n};\Z)$ is given by tensor
products of such generators, with bidegrees adding; a basis element is thus a pair
$(i_{1},i_{2})\in S_{m}\times S_{n}$ together with a choice of lower or upper generator
on each factor, and it lies in homological degree $i=i_{1}+i_{2}$.

Let $\varepsilon\in\{0,1,2\}$ be the number of factors on which the lower generator was
chosen. Each such choice lowers the quantum grading by $2$ and leaves $i$ unchanged, so
\[
q\ =\ (m+2i_{1})+(n+2i_{2})-2\varepsilon\ =\ (m+n)+2(i_{1}+i_{2})-2\varepsilon ,
\]
which compared with $q=(m+n)+2d$ gives $i=i_{1}+i_{2}=d+\varepsilon$. Hence in quantum
grading $q$ only the degrees $d,d+1,d+2$ occur, and the basis elements in degree
$d+\varepsilon$ are counted by the $\binom{2}{\varepsilon}$ choices of which factors are
lower, together with the $P(d+\varepsilon)$ pairs $(i_{1},i_{2})$ summing to
$d+\varepsilon$. This is the stated rank.

Freeness is \Cref{cor:kunneth-torus}, and the last sentence of the statement is
$\widetilde H^{\,i}\bigl(\Xo^{q}(L);\Z\bigr)\cong\Kho^{\,i,q}(L;\Z)$ together with
\Cref{conv:variance}.
\end{proof}

\begin{remark}\label{rem:diagonals}
A single factor has homology in at most two consecutive degrees, so its
spectrum is a wedge of spheres. For the split union, the outer degrees $d$
and $d+2$ can support the two cells of $\Sigma^{d-2}\CP^2$; the middle degree
contributes only spheres in the classification below.
\end{remark}

\begin{figure}[ht]
\centering
\begin{tikzpicture}[x=0.78cm,y=0.5cm]
\fill[black!18] (4,6) rectangle (7,7);
\fill[black!18] (11.4,6) rectangle (12.4,7);
\draw[black!30,step=1] (0,0) grid (11,13);
\draw (0,0) rectangle (11,13);
\foreach \r in {1,...,12} \draw[black!30] (11.4,\r) -- (12.4,\r);
\draw (11.4,0) rectangle (12.4,13);
\foreach \r in {0,...,12} {
  \pgfmathtruncatemacro{\dd}{\r-2}
  \node[font=\small] at (11.9,{\r+0.5}) {$\dd$};
}
\node[font=\small] at (11.9,13.55) {$d$};
\foreach \i/\p in {0/1,2/2,3/2,4/3,5/3,6/4,7/3,8/3,9/2,10/1} {
  \pgfmathtruncatemacro{\pp}{2*\p}
  \node[font=\small] at ({\i+0.5},{\i+0.5}) {$\p$};
  \node[font=\small] at ({\i+0.5},{\i+1.5}) {$\pp$};
  \node[font=\small] at ({\i+0.5},{\i+2.5}) {$\p$};
}
\foreach \i in {0,...,10}
  \node[font=\scriptsize] at ({\i+0.5},-0.55) {$\i$};
\foreach \r in {0,...,12} {
  \pgfmathtruncatemacro{\q}{6+2*\r}
  \node[font=\scriptsize,anchor=east] at (-0.1,{\r+0.5}) {$\q$};
}
\node[font=\small] at (5.5,-1.35) {homological grading $i$};
\node[font=\small,rotate=90] at (-1.25,6.5) {quantum grading $q$};
\end{tikzpicture}
\caption{The ranks of $\Kho^{\,i,q}(L_{4,6};\Z)$ (\Cref{lem:window}); empty boxes are $0$,
and the right-hand column gives $d$ for each row. In the shaded row $q=18$ we have $d=4$,
and the ranks $P(4),2P(5),P(6)$ are $3,6,4$. Read by columns, column $i$ contains $P(i)$, $2P(i)$, $P(i)$, so the homology lies
on three $\delta$-diagonals with multiplicities $1,2,1$; column $1$ is empty because
$1\notin S_k$.}
\label{fig:window}
\end{figure}

By \Cref{lem:window}, the groups $\Kho^{\,i,q}(L_{m,n};\Z)$ are determined by the numbers
$P(t)$. We record their zeros, and their values once $m,n\ge t$.

\begin{lemma}\label{lem:Pvanishing}
$S_m+S_n\subseteq\{0\}\cup[2,m+n]$. Hence $P(t)=0$ for $t\le-1$, for $t=1$, and for $t\ge m+n+1$. Moreover, if $m,n\ge t$ then
\[
P(t)=\#\bigl\{(i_1,i_2)\in(\Z_{\ge0}\setminus\{1\})^{2}: i_1+i_2=t\bigr\}=\begin{cases}1,&t=0,\\0,&t=1,\\ 2,&t=2,\\ t-1,&t\ge3.\end{cases}
\]
\end{lemma}

\begin{proof}
Let $(i_{1},i_{2})\in S_{m}\times S_{n}$. If $i_{1}=i_{2}=0$ then $i_{1}+i_{2}=0$;
otherwise at least one coordinate is $\ge2$, so $i_{1}+i_{2}\ge2$. Also $i_{1}\le m$ and
$i_{2}\le n$, so $i_{1}+i_{2}\le m+n$. This proves
$S_{m}+S_{n}\subseteq\{0\}\cup[2,m+n]$, and the three vanishing statements for $P$ follow
at once.

For the second statement, write $S_{k}=\{i\in\Z_{\ge0}\ :\ i\ne1\ \text{and}\ i\le k\}$.
If $i_{1}+i_{2}=t$ with $i_{1},i_{2}\ge0$ then $i_{1}\le t$ and $i_{2}\le t$, so as soon
as $m,n\ge t$ the constraints $i_{1}\le m$ and $i_{2}\le n$ are automatically satisfied.
The pairs counted by $P(t)$ are then exactly the pairs
$(i_1,i_2)\in(\Z_{\ge0}\setminus\{1\})^{2}$ with $i_1+i_2=t$, which is the first equality.

Finally we count these pairs for $t\ge0$. There are $t+1$ pairs of non-negative
integers summing to $t$, and from these we must discard the ones having a coordinate
equal to $1$. There is one pair with $i_{1}=1$, namely $(1,t-1)$, and one with $i_{2}=1$,
namely $(t-1,1)$, both present exactly when $t\ge1$; they coincide precisely when $t=2$.
Hence: for $t=0$ nothing is discarded and $P(0)=1$; for $t=1$ the two discarded
pairs $(1,0)$ and $(0,1)$ exhaust all $t+1=2$ pairs, so $P(1)=0$; for $t=2$ the
single pair $(1,1)$ is discarded from three, so $P(2)=2$; and for $t\ge3$ two
distinct pairs are discarded from $t+1$, so $P(t)=t-1$.
\end{proof}

\section{Spectra with free homology in three degrees}\label{sec:classification}

A finite spectrum with free integral homology in three consecutive degrees is
determined by those ranks and the rank of $\Sq^2$. We give the elementary
attaching-map argument, including the cofiber calculation needed in the skein
step. Throughout the proof for spaces, the assumption $d\ge3$ excludes
Whitehead products in the relevant degree.

\begin{lemma}\label{lem:cp2pi}
$\pi_{d+1}\bigl(\Sigma^{d-2}\CP^{2}\bigr)=0$.
\end{lemma}

\begin{proof}
Since $\Sigma^{d-2}\CP^{2}$ is the suspension spectrum of $\CP^{2}$, shifted so that its
cells lie in degrees $d$ and $d+2$, the group in question is the stable homotopy group
$\pi^{s}_{3}(\CP^{2})$.

First, $\pi_{3}(\CP^{2})=0$. Write
\[
S^{5}=\bigl\{(z_{0},z_{1},z_{2})\in\mathbb{C}^{3}\ :\ |z_{0}|^{2}+|z_{1}|^{2}+|z_{2}|^{2}=1\bigr\},
\qquad
S^{1}=\{\lambda\in\mathbb{C}\ :\ |\lambda|=1\},
\]
and let $S^{1}$ act on $S^{5}$ by
$\lambda\cdot(z_{0},z_{1},z_{2})=(\lambda z_{0},\lambda z_{1},\lambda z_{2})$. The action
is free, because the $z_{i}$ are not all zero, and its orbits are exactly the fibers of
\[
p\colon\ S^{5}\longrightarrow\CP^{2},\qquad
p(z_{0},z_{1},z_{2})=[z_{0}:z_{1}:z_{2}].
\]
So $p$ is a principal $S^{1}$-bundle, in particular a fibration
$S^{1}\to S^{5}\xrightarrow{\,p\,}\CP^{2}$, whose long exact sequence of homotopy groups
contains
\[
\pi_{3}(S^{5})\longrightarrow\pi_{3}(\CP^{2})\longrightarrow\pi_{2}(S^{1}).
\]
Both outer groups vanish, the first because $S^{5}$ is $4$-connected and the second
because $S^{1}$ is aspherical. Hence $\pi_{3}(\CP^{2})=0$.

Now we stabilize. If $X$ is $m$-connected then the suspension map
$\pi_{i}(X)\to\pi_{i+1}(\Sigma X)$ is an isomorphism for $i\le 2m$ and a surjection for
$i=2m+1$, by the Freudenthal suspension theorem \cite[\S4.2]{Hatcher}. The space $\CP^{2}$
is $1$-connected, so taking $m=1$ and $i=3$ the map
$\pi_{3}(\CP^{2})\to\pi_{4}(\Sigma\CP^{2})$ is surjective, and therefore
$\pi_{4}(\Sigma\CP^{2})=0$. For $k\ge1$ the space $\Sigma^{k}\CP^{2}$ is
$(k+1)$-connected, and $k+3\le 2(k+1)$, so
$\pi_{k+3}(\Sigma^{k}\CP^{2})\to\pi_{k+4}(\Sigma^{k+1}\CP^{2})$ is an isomorphism. The
groups $\pi_{k+3}(\Sigma^{k}\CP^{2})$ are therefore all isomorphic for $k\ge1$ and all
vanish, so $\pi^{s}_{3}(\CP^{2})=0$.
\end{proof}

We use the following standard fact about wedges. Let $A_{1},\dots,A_{r}$ be finite CW
complexes, each with a single $0$-cell and no cells in dimensions $1,\dots,n-1$. The product
$A_{1}\times\cdots\times A_{r}$ is obtained from $A_{1}\wed\cdots\wed A_{r}$ by attaching
cells of dimension at least $2n$, so the inclusion of the wedge into the product is
$(2n-1)$-connected. Composing with the projections of the product, the collapse maps
$\mathrm{pr}_{j}\colon A_{1}\wed\cdots\wed A_{r}\to A_{j}$ therefore induce isomorphisms
\begin{equation}\label{eq:wedgesum}
\pi_{i}(A_{1}\wed\cdots\wed A_{r})\ \xrightarrow{\ \cong\ }\ \bigoplus_{j=1}^{r}\pi_{i}(A_{j})
\qquad\text{for } i<2n-1 .
\end{equation}
By the Freudenthal suspension theorem these groups are stable for $i<2n-1$, so
\eqref{eq:wedgesum} also holds for the suspension spectra. We need it only for $i=n+1$, and
then the condition $i<2n-1$ says exactly $n\ge3$. In our applications
$n=d$ and $i=d+1$, which is why we assume $d\ge3$ in \Cref{lem:etacofiber} and \Cref{prop:classification}. For $n=2$ and $i=3$ it fails:
$\pi_{3}(\Sm^{2}\wed\Sm^{2})\cong\Z^{3}$, where the Whitehead product $[\iota_{1},\iota_{2}]$
generates a summand killed by both collapse maps.

\begin{lemma}\label{lem:etacofiber}
Let $d\ge3$ and $r\ge1$, and let $\eta\in\pi_{d+1}(\Sm^{d})=\pi^{s}_{1}\cong\Z/2$ be the
generator, the stable class of the Hopf map. Let $g\colon\Sm^{d+1}\to(\Sm^{d})^{\wed r}$ be a
map, and put $Y:=\Cone(g)$. For $1\le j\le r$ let
$\mathrm{pr}_{j}\colon(\Sm^{d})^{\wed r}\to\Sm^{d}$ collapse all summands but the $j$-th.
Each composite $\mathrm{pr}_{j}\circ g$ is $0$ or $\eta$, and by \eqref{eq:wedgesum} with $n=d$,
in degree $d+1$, $g\simeq0$ if and only if $\mathrm{pr}_{j}\circ g=0$ for every $j$. Let
$x_{1},\dots,x_{r}\in\widetilde H^{\,d}(Y;\F_{2})$ be the classes such that $x_{j}$
restricts to the generator on the $j$-th summand of $(\Sm^{d})^{\wed r}$ and to $0$ on the
others, and let $z\in\widetilde H^{\,d+2}(Y;\F_{2})$ be the generator. Then
\[
\Sq^{2}x_{j}\ =\
\begin{cases}
z, & \mathrm{pr}_{j}\circ g=\eta,\\
0, & \mathrm{pr}_{j}\circ g=0,
\end{cases}
\]
so $\Sq^{2}\colon\widetilde H^{\,d}(Y;\F_{2})\to\widetilde H^{\,d+2}(Y;\F_{2})$ is non-zero if
and only if $g\not\simeq0$. Consequently
\[
Y\ \simeq\
\begin{cases}
(\Sm^{d})^{\wed r}\wed\Sm^{d+2}, & g\simeq0,\\[2pt]
\Sigma^{d-2}\CP^{2}\wed(\Sm^{d})^{\wed r-1}, & g\not\simeq0 .
\end{cases}
\]
\end{lemma}

\begin{proof}
Fix $j$ and put $Y_{j}:=\Cone(\mathrm{pr}_{j}\circ g)$, so that
\[
Y=(\Sm^{d})^{\wed r}\cup_{g}D^{d+2},\qquad Y_{j}=\Sm^{d}\cup_{\mathrm{pr}_{j}\circ g}D^{d+2}.
\]
Since $\mathrm{pr}_{j}\circ g=(\mathrm{pr}_{j}\circ g)\circ\mathrm{id}$, there is a map
$f_{j}\colon Y\to Y_{j}$ which is $\mathrm{pr}_{j}$ on $(\Sm^{d})^{\wed r}$ and the identity on
$D^{d+2}$. With $i,i_{j}$ the inclusions of the bases and $q,q_{j}$ the maps collapsing them,
this gives a map of cofibration sequences
\[
\begin{tikzcd}
\Sm^{d+1}\arrow[r,"g"]\arrow[d,equal] & (\Sm^{d})^{\wed r}\arrow[r,"i"]\arrow[d,"\mathrm{pr}_{j}"] & Y\arrow[r,"q"]\arrow[d,"f_{j}"] & \Sm^{d+2}\arrow[d,equal]\\
\Sm^{d+1}\arrow[r,"\mathrm{pr}_{j}\circ g"'] & \Sm^{d}\arrow[r,"i_{j}"'] & Y_{j}\arrow[r,"q_{j}"'] & \Sm^{d+2}
\end{tikzcd}
\]
in which the right-hand vertical map is the identity because $f_{j}$ is the identity on
$D^{d+2}$.

All cohomology is with $\F_{2}$ coefficients. Let $u\in\widetilde H^{d+2}(\Sm^{d+2})$ and
$\iota\in\widetilde H^{d}(\Sm^{d})$ be the generators. The bases $(\Sm^{d})^{\wed r}$ and
$\Sm^{d}$ have no cohomology in degrees $d+1$ and $d+2$, and $\Sm^{d+1}$ and $\Sm^{d+2}$ have
none in degree $d$, so the long exact sequences of the two rows show that $q^{*}$ and
$q_{j}^{*}$ are isomorphisms in degree $d+2$, and that $i^{*}$ and $i_{j}^{*}$ are
isomorphisms in degree $d$. Thus $z=q^{*}u$; put $z_{j}:=q_{j}^{*}u$ and
$w_{j}:=(i_{j}^{*})^{-1}\iota$, and write $f_{j}^{d}$ and $f_{j}^{d+2}$ for the maps induced
by $f_{j}$ on $\widetilde H^{d}$ and on $\widetilde H^{d+2}$. Applying cohomology to the right-hand and the middle square of
the diagram above gives the commutative squares
\[
\begin{tikzcd}
\widetilde H^{d+2}(Y_{j})\arrow[d,"f_{j}^{d+2}"'] & \widetilde H^{d+2}(\Sm^{d+2})\arrow[l,"q_{j}^{*}"',"\cong"]\arrow[d,equal]\\
\widetilde H^{d+2}(Y) & \widetilde H^{d+2}(\Sm^{d+2})\arrow[l,"q^{*}","\cong"']
\end{tikzcd}
\qquad\qquad
\begin{tikzcd}
\widetilde H^{d}(Y_{j})\arrow[r,"i_{j}^{*}","\cong"']\arrow[d,"f_{j}^{d}"'] & \widetilde H^{d}(\Sm^{d})\arrow[d,"\mathrm{pr}_{j}^{*}"]\\
\widetilde H^{d}(Y)\arrow[r,"i^{*}","\cong"'] & \widetilde H^{d}\bigl((\Sm^{d})^{\wed r}\bigr)
\end{tikzcd}
\]
The first gives $f_{j}^{d+2}z_{j}=f_{j}^{d+2}q_{j}^{*}u=q^{*}u=z$. For the second, let $\mathrm{in}_{k}\colon\Sm^{d}\to(\Sm^{d})^{\wed r}$, for
$1\le k\le r$, be the inclusion of the $k$-th summand. The restrictions
$\mathrm{in}_{k}^{*}$ split the cohomology of the wedge as a direct sum,
\[
\widetilde H^{d}\bigl((\Sm^{d})^{\wed r}\bigr)\ \xrightarrow[\ \cong\ ]{\ (\mathrm{in}_{1}^{*},\dots,\mathrm{in}_{r}^{*})\ }\ \bigoplus_{k=1}^{r}\widetilde H^{d}(\Sm^{d}),
\]
so a class is determined by its restrictions to the summands. Since
$\mathrm{pr}_{j}\circ\mathrm{in}_{k}$ is the identity of $\Sm^{d}$ if $k=j$ and constant
otherwise, $\mathrm{in}_{k}^{*}\mathrm{pr}_{j}^{*}\iota$ is $\iota$ for $k=j$ and $0$ for
$k\ne j$. Thus $\mathrm{pr}_{j}^{*}\iota$ is the generator in the $j$-th coordinate of this
splitting, which by the definition of $x_{j}$ is $i^{*}x_{j}$. Hence
$i^{*}f_{j}^{d}w_{j}=\mathrm{pr}_{j}^{*}i_{j}^{*}w_{j}=\mathrm{pr}_{j}^{*}\iota=i^{*}x_{j}$,
and $f_{j}^{d}w_{j}=x_{j}$ since $i^{*}$ is injective.

Naturality of $\Sq^{2}$ along $f_{j}$ gives the commutative square
\[
\begin{tikzcd}
\widetilde H^{d}(Y_{j})\arrow[r,"\Sq^{2}"]\arrow[d,"f_{j}^{d}"'] & \widetilde H^{d+2}(Y_{j})\arrow[d,"f_{j}^{d+2}"]\\
\widetilde H^{d}(Y)\arrow[r,"\Sq^{2}"] & \widetilde H^{d+2}(Y)
\end{tikzcd}
\]
so $\Sq^{2}x_{j}=f_{j}^{d+2}\Sq^{2}w_{j}$. If $\mathrm{pr}_{j}\circ g=\eta$, then $Y_{j}=\Sigma^{d-2}\CP^{2}$ and $\Sq^{2}w_{j}=z_{j}$,
because $\Sq^{2}$ commutes with the suspension isomorphism
$\widetilde H^{*}(\CP^{2})\cong\widetilde H^{*+d-2}(\Sigma^{d-2}\CP^{2})$ and
$\Sq^{2}x=x^{2}\ne0$ for the generator $x\in H^{2}(\CP^{2})$; hence, applying
$f_{j}^{d+2}$, $\Sq^{2}x_{j}=f_{j}^{d+2}\Sq^{2}w_{j}=f_{j}^{d+2}z_{j}=z$. If $\mathrm{pr}_{j}\circ g=0$, then $Y_{j}=\Sm^{d}\wed\Sm^{d+2}$ and $w_{j}=\rho^{*}\iota$ for
the map $\rho\colon Y_{j}\to\Sm^{d}$ collapsing $S^{d+2}$, so $\Sq^{2}w_{j}=\rho^{*}\Sq^{2}\iota=0$ since
$\widetilde H^{d+2}(\Sm^{d})=0$; hence $\Sq^{2}x_{j}=f_{j}^{d+2}\Sq^{2}w_{j}=0$. Since the $x_{j}$ form a basis of $\widetilde H^{\,d}(Y)$, $\Sq^{2}$ is
non-zero if and only if $\mathrm{pr}_{j}\circ g=\eta$ for some $j$, that is, if and only if
$g\not\simeq0$.

If $g\simeq0$ then $Y\simeq(\Sm^{d})^{\wed r}\wed\Sm^{d+2}$. Suppose $g\not\simeq0$. By the Freudenthal suspension theorem the suspension
$\pi_{3}\bigl((\Sm^{2})^{\wed r}\bigr)\to\pi_{4}\bigl((\Sm^{3})^{\wed r}\bigr)$ is surjective,
since $(\Sm^{2})^{\wed r}$ is $1$-connected, and the suspensions
$\pi_{k+1}\bigl((\Sm^{k})^{\wed r}\bigr)\to\pi_{k+2}\bigl((\Sm^{k+1})^{\wed r}\bigr)$ are
isomorphisms for $k\ge3$. Hence $g=\Sigma^{d-2}\phi$ for a map of spaces
$\phi\colon\Sm^{3}\to(\Sm^{2})^{\wed r}$, and $Y\simeq\Sigma^{d-2}M$, where $M$ is the cone on
$\phi$. It is a $4$-dimensional CW-complex with a single $4$-cell and
\[
H_{1}(M;\Z)=0,\qquad H_{2}(M;\Z)=\Z^{r},\qquad H_{3}(M;\Z)=0,\qquad H_{4}(M;\Z)=\Z .
\]
This is the homology required in \cite[(11)]{SoTheriault}, with no torsion and with the
free ranks in degrees $1$ and $3$ equal to zero. Since $Y\simeq\Sigma^{d-2}M$ and $\Sq^{2}$
commutes with the suspension isomorphism, $\Sq^{2}$ acts non-trivially on $H^{*}(M;\F_{2})$.
Hence \cite[Thm.~5.9]{SoTheriault} applies, and with all Moore spaces and all $\Sm^{2}$ and
$\Sm^{4}$ summands absent it gives
\[
\Sigma M\ \simeq\ (\Sm^{3})^{\wed r-1}\wed\Sigma\CP^{2},
\]
hence $Y\simeq\Sigma^{d-3}\Sigma M\simeq(\Sm^{d})^{\wed r-1}\wed\Sigma^{d-2}\CP^{2}$.
\end{proof}
We first treat simply connected spaces with free homology in three consecutive degrees, for
$d\ge3$, and then pass to spectra in \Cref{cor:classification}.
\begin{proposition}\label{prop:classification}
Let $d\ge3$ and let $A$ be a simply connected finite CW complex whose reduced integral
homology is free and concentrated in degrees $d,d+1,d+2$, of ranks $a,b,c$. Then there is a
unique integer $\nu$ with $0\le\nu\le\min(a,c)$ such that
\[
A\ \simeq\ (\Sigma^{d-2}\CP^{2})^{\wed\nu}\ \wed\ (\Sm^{d})^{\wed a-\nu}\ \wed\ (\Sm^{d+1})^{\wed b}\ \wed\ (\Sm^{d+2})^{\wed c-\nu}.
\]
Moreover
$\nu=\rk_{\F_2}\bigl(\Sq^{2}\colon \widetilde H^{\,d}(A;\F_2)\to\widetilde H^{\,d+2}(A;\F_2)\bigr)$.
\end{proposition}

\begin{proof}
As in the proof of \Cref{lem:factorspectra}, \cite[Prop.~4C.1]{Hatcher} gives a CW model of
$A$ with one $0$-cell and $a$, $b$, $c$ cells of dimensions $d$, $d+1$, $d+2$; its cellular
differentials vanish, and the attaching maps of the $(d+1)$-cells are null. So we may take
the $(d+1)$-skeleton to be
\[
W=(\Sm^{d})^{\wed a}\wed(\Sm^{d+1})^{\wed b}.
\]
Attach the $(d+2)$-cells one at a time, and let $A_k$ consist of $W$ and the first $k$ of
them. We prove inductively that
\begin{equation}\label{eq:Xk}
A_k\simeq(\Sigma^{d-2}\CP^{2})^{\wed\nu_k}\wed(\Sm^{d})^{\wed a-\nu_k}\wed(\Sm^{d+1})^{\wed b}\wed(\Sm^{d+2})^{\wed k-\nu_k},
\end{equation}
where $\nu_0=0$ and $\nu_k-\nu_{k-1}\in\{0,1\}$.

The next attaching map $\psi_k\colon\Sm^{d+1}\to W\subset A_{k-1}$ is zero on integral
homology, since $(\psi_k)_*[\Sm^{d+1}]\in\widetilde H_{d+1}(W;\Z)$ is the cellular boundary of
the new cell. Regard it, via \eqref{eq:Xk} for $k-1$, as a map into that wedge. By
\eqref{eq:wedgesum} it is determined by its components on the summands: those on
$\Sm^{d+1}$ vanish because they are degrees, those on $\Sm^{d+2}$ by connectivity, and those
on $\Sigma^{d-2}\CP^{2}$ by \Cref{lem:cp2pi}. So $\psi_k$ factors up to homotopy through the
remaining $d$-spheres. By \Cref{lem:etacofiber}, if this map is zero the new cell adds a
sphere $\Sm^{d+2}$, and otherwise it combines with one $d$-sphere into
$\Sigma^{d-2}\CP^{2}$; the other summands are unchanged. If no $d$-sphere is left, the map
is zero. This proves \eqref{eq:Xk}, and $k=c$ gives the decomposition.

Since $\Sq^{2}$ is non-zero on $\widetilde H^{d}(\Sigma^{d-2}\CP^{2};\F_2)$ and zero on
spheres, $\Sq^{2}\colon\widetilde H^{d}(A;\F_2)\to\widetilde H^{d+2}(A;\F_2)$ has rank
$\nu_c$. This proves uniqueness, and $\nu_c\le\min(a,c)$ because these groups have
dimensions $a$ and $c$.
\end{proof}

\begin{corollary}\label{cor:classification}
Let $X$ be a finite spectrum whose integral homology is free and concentrated in three
consecutive degrees $d,d+1,d+2$, of ranks $a,b,c$. Then there is a unique integer $\nu$
with $0\le\nu\le\min(a,c)$ such that
\[
X\ \simeq\ (\Sigma^{d-2}\Sigma^{\infty}\CP^{2})^{\wed\nu}\ \wed\ (\Sph^{d})^{\wed a-\nu}\ \wed\ (\Sph^{d+1})^{\wed b}\ \wed\ (\Sph^{d+2})^{\wed c-\nu}.
\]
Moreover
$\nu=\rk_{\F_2}\bigl(\Sq^{2}\colon \widetilde H^{\,d}(X;\F_2)\to\widetilde H^{\,d+2}(X;\F_2)\bigr)$.
\end{corollary}

\begin{proof}
Suspension is an equivalence of spectra that raises $d$ by one, and both it and
$\Sigma^{\infty}$ preserve wedges and commute with $\Sq^{2}$. Since $X$ is finite, after
sufficiently many suspensions we have $d\ge3$ and $X=\Sigma^{\infty}A$ for a simply
connected finite CW complex $A$, which satisfies the hypotheses of
\Cref{prop:classification}. Applying $\Sigma^{\infty}$ to the splitting of $A$ gives that
of $X$, and the last paragraph of the proof of \Cref{prop:classification} applies verbatim
to $X$, giving the uniqueness of $\nu$ and the formula for it.
\end{proof}

From now on we use the notation of spaces for the corresponding spectra: we write $\Sm^{k}$
for the sphere spectrum $\Sph^{k}$ and $\Sigma^{k}\CP^{2}$ for $\Sigma^{k}\Sigma^{\infty}\CP^{2}$,
so that \Cref{cor:classification} has the same form as \Cref{prop:classification}, without
the restriction $d\ge3$. Both are (de)suspensions of suspension spectra, so no confusion with
the corresponding spaces arises; when a space is meant, we say so.

\begin{definition}\label{def:nu}
For $m,n\ge2$ and $q\equiv m+n\pmod2$ let $\nu(m,n;q)$ denote the integer supplied by \Cref{cor:classification} for $X=\Xo^{q}(L_{m,n})$; by \Cref{lem:window} this gives \eqref{eq:intro-classification}.
\end{definition}

\begin{corollary}\label{cor:room}
For all $m,n\ge2$ and all $q\equiv m+n\pmod 2$,
\[
\nu(m,n;q)\ \le\ \min\bigl\{P(d),\,P(d+2)\bigr\}.
\]
\end{corollary}

\begin{proof}
By \Cref{lem:window} the spectrum $\Xo^{q}(L_{m,n})$ satisfies the hypotheses of
\Cref{cor:classification} with $(a,b,c)=\bigl(P(d),\,2P(d+1),\,P(d+2)\bigr)$, and that
corollary bounds $\nu$ by $\min(a,c)$.
\end{proof}
We now prove \Cref{thm:vanishing}.
\begingroup
\renewcommand{\thetheorem}{A}\renewcommand{\theHtheorem}{thmA}
\begin{theorem}\label{thm:vanishing}
Let $m,n\ge2$. If
\[
d\in\{-2,\,-1,\,1,\,m+n-1,\,m+n\},
\]
equivalently
\[
q\in\bigl\{m{+}n{-}4,\ m{+}n{-}2,\ m{+}n{+}2,\ 3(m{+}n){-}2,\ 3(m{+}n)\bigr\},
\]
then $\nu(m,n;q)=0$, so $\Xo^{q}(L_{m,n})$ is a wedge of spheres.
\end{theorem}
\addtocounter{theorem}{-1}\endgroup

\begin{proof}
The two lists correspond under $q=m+n+2d$. For each listed value,
\Cref{lem:Pvanishing} gives $P(d)=0$ or $P(d+2)=0$, so
\Cref{cor:room} gives $\nu=0$.
\end{proof}

\begin{remark}\label{rem:fivegradings}
The spectrum is contractible outside $-2\le d\le m+n$.
Within this range, the listed five values are precisely those for which
$P(d)P(d+2)=0$. The interior exception $d=1$ comes from the missing
homological degree $1$ in each factor.
\end{remark}

\section{The skein cofiber sequence and its source}\label{sec:skein}

\subsection{The cofibration}\label{ss:cofibration}

Let $D$ be an oriented diagram with a distinguished positive crossing $x$, and let $D_0$
and $D_1$ be the $0$- and $1$-resolution of $D$ at $x$. Splitting the cube of resolutions along the coordinate $x$ gives
a cofibration sequence of signed Burnside functors (\Cref{fig:cube}). On chain level it is a
short exact sequence
\begin{equation}\label{eq:skein-ses}
0\longrightarrow \mathit{CKh}_{o}(D_{1})\longrightarrow \mathit{CKh}_{o}(D)
\longrightarrow \mathit{CKh}_{o}(D_{0})\longrightarrow 0
\end{equation}
of bigraded complexes, up to shifts of the two gradings; after realization in spectra it yields \eqref{eq:skein} below,
the odd analogue of the unoriented skein cofibration of \cite{LS14a}.

\begin{figure}[ht]
\centering
\begin{tikzpicture}[>=stealth,every node/.style={font=\scriptsize}]
\foreach \mode/\xo in {0/0,2/4.3,1/8.6} {
\begin{scope}[xshift=\xo cm]
  \foreach \e/\h in {0/0,1/1.9} {
    \coordinate (A\e) at (0,\h);
    \coordinate (B\e) at (1.6,\h);
    \coordinate (C\e) at (0.65,{\h+0.6});
    \coordinate (D\e) at (2.25,{\h+0.6});
  }
  \foreach \e in {0,1} {
    \pgfmathtruncatemacro{\on}{(\mode==2) || (\mode==\e)}
    \ifnum\on=1
      \foreach \s/\t in {D/B,D/C,B/A,C/A}
        \draw[->,shorten >=2.6pt,shorten <=2.6pt] (\s\e) -- (\t\e);
    \else
      \foreach \s/\t in {D/B,D/C,B/A,C/A}
        \draw[black!25] (\s\e) -- (\t\e);
    \fi
  }
  \ifnum\mode=2
    \foreach \P in {A,B,C,D}
      \draw[->,thick,shorten >=2.6pt,shorten <=2.6pt] (\P1) -- (\P0);
  \else
    \foreach \P in {A,B,C,D}
      \draw[black!25] (\P1) -- (\P0);
  \fi
  \foreach \e in {0,1} {
    \pgfmathtruncatemacro{\on}{(\mode==2) || (\mode==\e)}
    \foreach \P in {A,B,C,D} {
      \ifnum\on=1 \fill (\P\e) circle (1.7pt);
      \else \draw[fill=white] (\P\e) circle (1.7pt); \fi
    }
  }
\end{scope}
}
\node[anchor=east] at (-0.15,0.3) {$\epsilon=0$};
\node[anchor=west] at (11.0,2.2) {$\epsilon=1$};
\node[font=\Large] at (3.25,1.25) {$\hookrightarrow$};
\node[font=\Large] at (7.55,1.25) {$\twoheadrightarrow$};
\node at (1.1,-0.55) {$G$};
\node at (5.4,-0.55) {$F=F^{q}_{o}(D)$};
\node at (9.7,-0.55) {$H$};
\node at (1.1,-1.25) {$\Xo^{\,q-1}(D_{0})$};
\node at (3.25,-1.25) {$\longrightarrow$};
\node at (5.4,-1.25) {$\Xo^{q}(D)$};
\node at (7.55,-1.25) {$\longrightarrow$};
\node at (9.7,-1.25) {$\Sigma^{c+1}\Xo^{\,q-3c-2}(D_{1})$};
\node at (6.8,1.55) {$\delta$};
\end{tikzpicture}
\caption{The cofibration sequence $G\hookrightarrow F\twoheadrightarrow H$ of signed Burnside
functors in the proof of \Cref{lem:skein}, drawn for $n=3$; in each cube the last
coordinate is $\epsilon$, the coordinate of the crossing $x$, and the face $\epsilon=0$ is
at the bottom. A filled dot is a vertex $v$ with its set $F(v)$, an open dot is a vertex
sent to $\varnothing$, and black arrows are the edges on which the functor is nonempty.
Thus $G$ and $H$ are the functors of $D_{0}$ and $D_{1}$ on the two faces, and the vertical
arrows of $F$ give the connecting map $\delta$. The bottom row is the realization; rotating
it gives \eqref{eq:skein}.}
\label{fig:cube}
\end{figure}
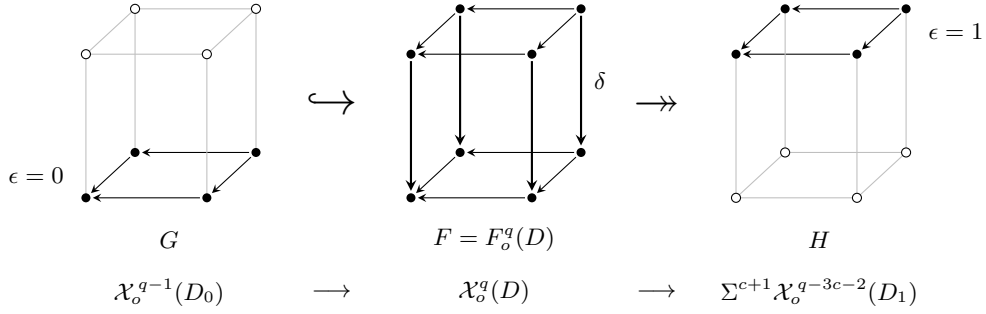

\noindent
The even counterpart of \eqref{eq:skein} is proved by exhibiting the flow category of
$D_0$ as a downward closed subcategory of that of $D$;
we give the argument for the odd theory in the language of signed Burnside functors, where
the corresponding statements are \Cref{def:cofseq} and \Cref{lem:realization}(i).

\begin{lemma}\label{lem:skein}
Let $D$ be an oriented link diagram with $n$ crossings and a distinguished \emph{positive}
crossing $x$. Let $D_{0}$ and $D_{1}$ be the $0$- and $1$-resolution of $D$ at $x$. Since
$x$ is positive, $D_{0}$ is the oriented resolution and carries the orientation induced from
$D$; equip $D_{1}$ with an arbitrary orientation. Put
\[
c:=n_{-}(D_{1})-n_{-}(D).
\]
Then for every $q$ there is a cofibration sequence of spectra
\begin{equation}\label{eq:skein}
\Sigma^{c}\,\Xo^{\,q-3c-2}(D_{1})\ \xrightarrow{\ \delta\ }\
\Xo^{\,q-1}(D_{0})\ \longrightarrow\ \Xo^{\,q}(D).
\end{equation}
\end{lemma}

\begin{proof}
Order the crossings of $D$ so that $x$ is the last one, write a vertex of the cube as
$v=(u,\epsilon)$ with $u\in\underline{2}^{\,n-1}$ and $\epsilon\in\{0,1\}$, and let
$F:=F^{q}_{o}(D)$.

Define $G$ by $G(v)=F(v)$ if
$\epsilon=0$ and $G(v)=\varnothing$ if $\epsilon=1$, and on morphisms by
$G(\phi_{v',v''})=F(\phi_{v',v''})$ when $v'$ and $v''$ both have last coordinate $0$ and
$G(\phi_{v',v''})=\varnothing$ otherwise. Conditions (i) and (ii) of \Cref{def:cofseq}
hold by construction. For (iii), let $v'\ge v''$ and $y\in G(v')$, and let $s\colon F(\phi_{v',v''})\to F(v')$ be
the source map of $F$. Since $G(v')\neq\varnothing$, the last coordinate of $v'$ is $0$, and
since $v''\le v'$, so is that of $v''$. Hence $G(\phi_{v',v''})=F(\phi_{v',v''})$, which
contains $s^{-1}(y)$. So $G$ is a sub-functor of
$F$, and with its quotient $H$ (\Cref{def:cofseq}) we obtain a cofibration sequence of
signed Burnside functors
\[
G\ \longrightarrow\ F\ \longrightarrow\ H .
\]
By \Cref{lem:realization}(i) realization carries it to a cofibration sequence of spectra.

By \Cref{def:cofseq}, $H(v)=F(v)\smallsetminus G(v)$ is $F(v)$ if $\epsilon=1$ and
$\varnothing$ if $\epsilon=0$, and $H(\phi_{v',v''})=t^{-1}\bigl(H(v'')\bigr)$, where
$t\colon F(\phi_{v',v''})\to F(v'')$ is the target map of $F$, is all of $F(\phi_{v',v''})$
if $v''$ has last coordinate $1$ (and then so has $v'\ge v''$), and $\varnothing$ otherwise.
So $H$ is $F$ on the face $\{v:\epsilon=1\}$ and empty elsewhere, just as $G$ is $F$ on the
face $\{v:\epsilon=0\}$.

Deleting the last coordinate identifies both faces with $\underline{2}^{\,n-1}$, the cube of
$D_{\epsilon}$, whose crossings are those of $D$ other than $x$: the resolution
$D_{(u,\epsilon)}$ is $(D_{\epsilon})_{u}$, and an edge inside a face is the surgery at a
crossing of $D_{\epsilon}$. The signed Burnside data at a vertex and at an edge depend only on
the resolution there, on the surgery arc and on the auxiliary choices of \cite[\S5.3]{SSS}
(orientations at the crossings, orderings of the circles, an edge assignment), and the
choices made for $D$ restrict to choices for $D_{\epsilon}$. Hence on their faces $G$ and $H$
are the signed Burnside functors of $D_{0}$ and of $D_{1}$, up to a shift of the two
gradings. The
realization therefore has the form
\begin{equation}\label{eq:skein-shape}
\Sigma^{a}\Xo^{\,q+b}(D_{0})\ \longrightarrow\ \Xo^{\,q}(D)\ \longrightarrow\
\Sigma^{a'}\Xo^{\,q+b'}(D_{1}),
\end{equation}
where $a,b,a',b'$ are integers depending only on $D$ and $x$, not on $q$.

On cohomology, \eqref{eq:skein-shape} induces the odd skein
long exact sequence of \cite[Prop.~1.6]{ORS}. The two gradings of \Cref{ss:oddkh} are given
by the same formulae as in the even theory, so its shifts are the
ones recorded in \cite[Prop.~4.2]{Ras05} and \cite[Prop.~4]{MO} --- the sources to which
\cite{ORS} refer for precisely this point. With $x$ positive and $D_{0}$ the oriented
resolution, those read
\[
\cdots\to \Kho^{\,i-c-1,\,q-3c-2}(D_{1})\to \Kho^{\,i,q}(D)\to \Kho^{\,i,q-1}(D_{0})
\xrightarrow{\delta^{*}} \Kho^{\,i-c,\,q-3c-2}(D_{1})\to\cdots
\]
Comparing this with the long exact sequence of \eqref{eq:skein-shape}, in which the three
terms contribute $\Kho^{\,i-a,\,q+b}(D_{0})$, $\Kho^{\,i,q}(D)$ and
$\Kho^{\,i-a',\,q+b'}(D_{1})$, gives $a=0$, $b=-1$, $a'=c+1$ and $b'=-3c-2$. Rotating the
resulting cofibration sequence $A\to X\to C$ into $\Sigma^{-1}C\to A\to X$ yields
\eqref{eq:skein}, and the map $\delta^{*}$ in the long exact sequence above is induced
by $\delta$.
\end{proof}

\begin{remark}\label{rem:skeinneg}
For a negative crossing, $D_1$ inherits the orientation of $D$. The same
face construction and grading calculation give
\[
\Sigma^{-1}\Xo^{\,q+1}(D_{1})\ \longrightarrow\ \Sigma^{c'}\Xo^{\,q-3c'-1}(D_{0})\
\longrightarrow\ \Xo^{\,q}(D),\qquad c'=n_{-}(D_{0})-n_{-}(D).
\]
It is used in \Cref{sec:mirror} for the mirrors $T(2,-m)\sqcup T(2,-n)$.
\end{remark}

\begin{lemma}\label{lem:deloop}
Let $L$ be a link and $U$ the unknot. Then for every $q$,
\begin{equation}\label{eq:deloop}
\Xo^{q}(L\sqcup U)\ \simeq\ \bigvee_{q_{1}+q_{2}=q}\Xo^{q_{1}}(L)\wedge\Xo^{q_{2}}(U)
\ \simeq\ \Xo^{\,q-1}(L)\ \wed\ \Xo^{\,q+1}(L).
\end{equation}
\end{lemma}

\begin{proof}
The odd Khovanov spectrum of the unknot is
$\Xo^{\,q_{2}}(U)\simeq\Sm^{0}$ for $q_{2}=\pm1$ and $\Xo^{\,q_{2}}(U)\simeq\ast$ otherwise.
So only the two summands with $q_{2}=\pm1$ survive on the right-hand side, and since $\Sm^{0}$
is the unit for the smash product,
\[
\bigvee_{q_{1}+q_{2}=q}\Xo^{\,q_{1}}(L)\wedge\Xo^{\,q_{2}}(U)
\ \simeq\ \big(\Xo^{\,q-1}(L)\wedge\Sm^{0}\big)\vee\big(\Xo^{\,q+1}(L)\wedge\Sm^{0}\big)
\ \simeq\ \Xo^{\,q-1}(L)\vee\Xo^{\,q+1}(L).
\]

For the left-hand side we use the observation of \cite[\S5.5]{SSS}. Let $L$ have $n$
crossings and let $F\colon\underline{2}^{\,n}\to\mathscr B_{\sigma}$ be its signed Burnside
functor. Since $U$ is crossingless and disjoint from $L$, the diagram $L\sqcup U$ has the
same $n$ crossings; its functor $E$ therefore has the same source $\underline{2}^{\,n}$,
and at every vertex $u$ the resolution of $L\sqcup U$ is that of $L$ together with one
extra circle. Fix an ordering of the circles in which that circle comes last, and call it
$a$.

The generators of $E(u)$ are the $y\in F(u)$, which do not involve $a$, and the $y\wedge a$
with $y\in F(u)$. Adjoining $a$ adds one circle and leaves $|u|$, $n$ and $n_{-}$ unchanged,
so the quantum degree of $y$ rises by one while that of $y\wedge a$ falls by one. Read in a
fixed quantum degree $q$, this says that a generator $y$ of $E^{q}(u)$ has degree $q-1$ in
$F$, whereas for a generator $y\wedge a$ of $E^{q}(u)$ the element $y$ has degree $q+1$ in
$F$. Hence
\begin{equation}\label{eq:deloop-obj}
E^{\,q}(u)\ \cong\ \underbrace{F^{\,q-1}(u)}_{\text{not involving }a}\ \sqcup\
\underbrace{F^{\,q+1}(u)\wedge a}_{\text{involving }a},
\end{equation}
the second summand being identified with $F^{\,q+1}(u)$ by $y\mapsto y\wedge a$.

Let now $u\ge v$ in $\underline{2}^{\,n}$. Because $U$ is disjoint from $L$ and carries no
crossings, no edge of the cube merges $a$ with another circle or splits it; every edge acts
on the circles of $L$ alone. The morphism $E^{\,q}(\phi_{u,v})$ therefore preserves the
splitting \eqref{eq:deloop-obj}, and under the two identifications it is
$F^{\,q-1}(\phi_{u,v})\sqcup F^{\,q+1}(\phi_{u,v})$. Consequently
\[
E^{\,q}\ \cong\ F^{\,q-1}\ \sqcup\ F^{\,q+1}
\]
as signed Burnside functors $\underline{2}^{\,n}\to\mathscr B_{\sigma}$.

Since $L\sqcup U$ has the same negative crossings as $L$, \Cref{def:oddspectrum} and
\Cref{lem:realization}(ii) with $r=-n_{-}$ give
\[
\Xo^{q}(L\sqcup U)=\bigl|\Sigma^{-n_{-}}E^{q}\bigr|\ \simeq\
\bigl|\Sigma^{-n_{-}}F^{q-1}\bigr|\vee\bigl|\Sigma^{-n_{-}}F^{q+1}\bigr|
=\Xo^{\,q-1}(L)\vee\Xo^{\,q+1}(L),
\]
which is the right-hand side computed above.
\end{proof}

\begin{remark}
\Cref{lem:deloop} is a case in which the smash formula \eqref{eq:even-smash} \emph{does}
hold for the odd theory: the spectrum of the unknot is a wedge of two spheres, so smashing
with it only reindexes. For split unions in general the formula fails
(\Cref{ss:earlier}).
\end{remark}

We now specialize \Cref{lem:skein} to the links $L_{m,n}$, the only case used below.

\begin{corollary}\label{cor:skeinLmn}
Let $m\ge1$ and $n\ge2$, let $D$ be the standard diagram of $L_{m,n}$, and resolve a
crossing $x$ of its $T(2,n)$ factor (\Cref{fig:skeinLmn}). Then $D_{0}$ is the standard
diagram of $L_{m,n-1}$, $D_{1}$ is a diagram of $T(2,m)\sqcup U$, and $c=n-1$ when
$T(2,m)$ keeps the orientation of $D$. Hence for every $q$ there is a cofibration sequence
\begin{equation}\label{eq:skeinLmn}
\Sigma^{n-1}\Xo^{\,q-3n+1}\bigl(T(2,m)\sqcup U\bigr)\ \xrightarrow{\ \delta\ }\
\Xo^{\,q-1}(L_{m,n-1})\ \longrightarrow\ \Xo^{\,q}(L_{m,n}),
\end{equation}
where $\Xo^{\,q-3n+1}\bigl(T(2,m)\sqcup U\bigr)\simeq\Xo^{\,q-3n}\bigl(T(2,m)\bigr)\vee
\Xo^{\,q-3n+2}\bigl(T(2,m)\bigr)$ by \Cref{lem:deloop}.
\end{corollary}

\begin{proof}
In the standard diagram both strands of each twist region point the same way, so every
crossing is positive and $n_{-}(D)=0$. The $0$-resolution at $x$, the oriented one, is the
twist region with one crossing fewer. The $1$-resolution joins the two strands of the
$T(2,n)$ factor into one unknotted component, along which they run in opposite directions;
so its remaining $n-1$ crossings are negative, while those of $T(2,m)$ stay positive. Hence
$n_{-}(D_{1})=n-1$ and $c=n-1$, and \Cref{lem:skein} gives \eqref{eq:skeinLmn}.
\end{proof}

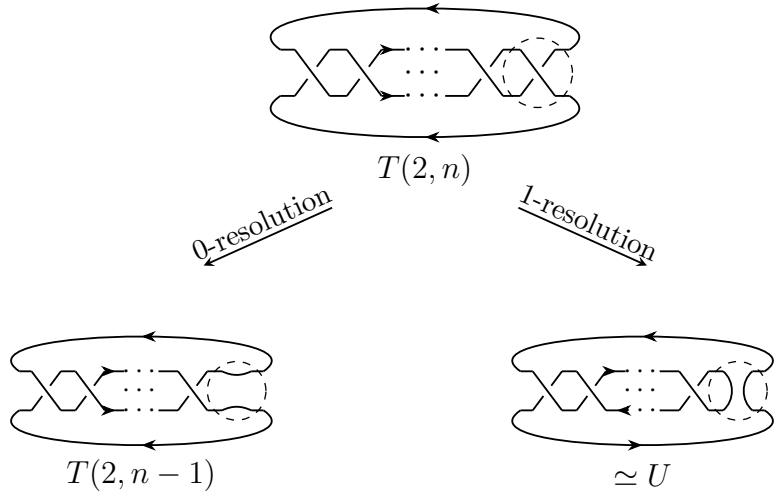
\begin{figure}[ht]
\centering
\def\tilecr#1{\draw (#1+0.6,-0.8) -- (#1-0.6,0.8);
  \draw (#1-0.6,-0.8) -- (#1-0.16,-0.21);
  \draw (#1+0.16,0.21) -- (#1+0.6,0.8);}
\def\tilezero#1{\draw (#1-0.6,0.8)  to[out=-25,in=-155] (#1+0.6,0.8);
  \draw (#1-0.6,-0.8) to[out=25,in=155]   (#1+0.6,-0.8);}
\def\tileone#1{\draw (#1-0.6,0.8) to[out=-40,in=40]   (#1-0.6,-0.8);
  \draw (#1+0.6,0.8) to[out=-140,in=140] (#1+0.6,-0.8);}
\def\braid#1#2#3{\draw[dashed,line width=0.5pt] (7.8,0) circle (1.2);
  \tilecr{0}\tilecr{1.8}\tilecr{6.0}#1{7.8}
  \draw (0.6,0.8) -- (1.2,0.8);    \draw (0.6,-0.8) -- (1.2,-0.8);
  \draw[midarr=0.55] (2.4,0.8) -- (3.2,0.8);
  \draw[#3=0.55]     (2.4,-0.8) -- (3.2,-0.8);
  \draw (4.6,0.8) -- (5.4,0.8);    \draw (4.6,-0.8) -- (5.4,-0.8);
  \draw (6.6,0.8) -- (7.2,0.8);    \draw (6.6,-0.8) -- (7.2,-0.8);
  \node[#2] at (3.9,0.8){$\cdots$}; \node[#2] at (3.9,0){$\cdots$};
  \node[#2] at (3.9,-0.8){$\cdots$};
  \draw (8.4,0.8) -- (8.9,0.8);    \draw (8.4,-0.8) -- (8.9,-0.8);
  \draw (-1.1,0.8) -- (-0.6,0.8);  \draw (-1.1,-0.8) -- (-0.6,-0.8);
  \draw[midarr] (8.9,0.8)  .. controls (11.7,2.7)  and (-3.9,2.7)  .. (-1.1,0.8);
  \draw[#3]     (8.9,-0.8) .. controls (11.7,-2.7) and (-3.9,-2.7) .. (-1.1,-0.8);}
\begin{tikzpicture}[scale=0.85,line width=0.7pt,>=stealth,
  midarr/.style={postaction={decorate},decoration={markings,
    mark=at position #1 with {\arrow{Stealth[length=5.5pt]}}}},
  midarr/.default=0.5,
  midarrr/.style={postaction={decorate},decoration={markings,
    mark=at position #1 with {\arrowreversed{Stealth[length=5.5pt]}}}},
  midarrr/.default=0.5]
\begin{scope}[shift={(-1.76,2.70)},scale=0.45]
  \braid{\tilecr}{}{midarr}
  \node at (3.9,-3.35) {$T(2,n)$};
\end{scope}
\draw[->] (-1.45,0.60) -- node[sloped,above,inner sep=2pt] {\small $0$-resolution} (-3.45,-0.30);
\begin{scope}[shift={(-5.90,-2.25)},scale=0.38]
  \braid{\tilezero}{scale=0.9}{midarr}
  \node at (3.9,-3.45) {$T(2,n-1)$};
\end{scope}
\draw[->] (1.45,0.60) -- node[sloped,above,inner sep=2pt] {\small $1$-resolution} (3.45,-0.30);
\begin{scope}[shift={(1.90,-2.25)},scale=0.38]
  \braid{\tileone}{scale=0.9}{midarrr}
  \node at (3.9,-3.45) {$\simeq U$};
\end{scope}
\end{tikzpicture}
\caption{Resolving the circled crossing of the standard diagram of $T(2,n)$, as in
\Cref{cor:skeinLmn}. The $0$-resolution is the standard diagram of $T(2,n-1)$; the
$1$-resolution is a diagram of the unknot $U$, along which the two strands of the twist
region run in opposite directions.}
\label{fig:skeinLmn}
\end{figure}

\subsection{The connecting map vanishes on cohomology}

\begin{lemma}\label{lem:deltastar}
Let $m,n\ge2$ and let $D$ be as in \Cref{cor:skeinLmn}. Then the map $\delta^{*}$ induced
on integral cohomology by the connecting map $\delta$ of \eqref{eq:skeinLmn} is zero.
\end{lemma}

\begin{proof}
By \eqref{eq:kunneth-chain} the odd Khovanov complex of a split union is the tensor
product of the two complexes, so the short exact sequence \eqref{eq:skein-ses} for $D$ is
obtained from the corresponding sequence for $T(2,n)$ alone by
applying $\displaystyle{\mathit{CKh}_{o}\bigl(T(2,m)\bigr)\otimes_{\Z}-}$. That complex is degreewise
free, hence flat, so the result is again exact; and its cohomology
$\Kho\bigl(T(2,m);\Z\bigr)$ is free by \Cref{lem:torusknot}, so the K\"unneth theorem
gives $ H^{*}\bigl(\mathit{CKh}_{o}(T(2,m))\otimes_{\Z}-\bigr)\cong
\Kho\bigl(T(2,m);\Z\bigr)\otimes_{\Z}H^{*}(-)$, naturally. The skein exact sequence of
\eqref{eq:skeinLmn} for $D$ is therefore $\Kho(T(2,m);\Z)\otimes(\text{the skein exact sequence
for }T(2,n))$, and in particular $\delta^{*}=\mathrm{id}\otimes\delta^{*}_{T(2,n)}$. It suffices
to treat $T(2,n)$ itself, for which $D_{0}=T(2,n-1)$, $D_{1}=U$ and $c=n-1$.
On cohomology \Cref{lem:skein} induces, writing $T_{k}:=T(2,k)$,
\begin{equation}\label{eq:deltaLES}
\cdots\to\Kho^{\,i,j}(T_{n})\to\Kho^{\,i,j-1}(T_{n-1})
\xrightarrow{\ \delta^{*}\ }\Kho^{\,i-n+1,\,j-3n+1}(U)
\xrightarrow{\ \alpha\ }\Kho^{\,i+1,j}(T_{n})\to\cdots
\end{equation}

Since $\Kho^{\,p,q}(U;\Z)=\Z$ for
$(p,q)=(0,1)$ and $(0,-1)$ and vanishes otherwise, the target of $\delta^{*}$ is non-zero
only when $i-n+1=0$ and $j-3n+1=\pm1$, that is, only for
\[
(i,j)=(n-1,\,3n)\qquad\text{and}\qquad(i,j)=(n-1,\,3n-2).
\]
In every other bidegree $\delta^{*}=0$ because its target vanishes.

\emph{Case 1: the bidegree $(n-1,3n)$.} Here the source of $\delta^{*}$ is
$\Kho^{\,n-1,\,3n-1}\bigl(T(2,n-1)\bigr)$. By \Cref{lem:torusknot},
$\Kho^{\,i,q}\bigl(T(2,k)\bigr)$ is non-zero only for $i\in S_{k}$ and
$q\in\{k+2i-2,\;k+2i\}$; with $k=i=n-1$ this reads $q\in\{3n-5,\,3n-3\}$, so the source
vanishes and $\delta^{*}=0$.

\emph{Case 2: the bidegree $(n-1,3n-2)$.} Here the source need not vanish, so we use the
next map in \eqref{eq:deltaLES} instead. Substituting $i=n-1$ and $j=3n-2$, the sequence
continues
\[
\Kho^{\,0,-1}(U)\ \xrightarrow{\ \alpha\ }\ \Kho^{\,n,\,3n-2}\bigl(T(2,n)\bigr)\
\longrightarrow\ \Kho^{\,n,\,3n-3}\bigl(T(2,n-1)\bigr).
\]
The right-hand group vanishes because $n\notin S_{n-1}$, and the middle group is $\Z$
because $n\in S_{n}$ and $3n-2=n+2n-2$. By exactness $\alpha$ is a surjection
$\Z\twoheadrightarrow\Z$, hence an isomorphism, and in particular injective. Exactness at
$\Kho^{\,0,-1}(U)$ then gives $\operatorname{im}\delta^{*}=\ker\alpha=0$, so $\delta^{*}=0$
here as well.
\end{proof}
\begin{lemma}\label{lem:source}
Let $m,n\ge2$, and let
\[
A:=\Sigma^{n-1}\Xo^{\,q-3n+1}\bigl(T(2,m)\sqcup U\bigr),\qquad
T:=\Xo^{\,q-1}(L_{m,n-1})
\]
be the two left-hand terms of \eqref{eq:skeinLmn}. Then:
\begin{enumerate}[label=\textup{(\roman*)}]
\item $\widetilde H^{\,*}(T;\Z)$ is free and concentrated in degrees $d,d+1,d+2$; and
\item $A$ is a wedge of spheres,
\[
A\ \simeq\ (\Sm^{d-1})^{\wed\alpha}\wed(\Sm^{d})^{\wed\beta}\wed(\Sm^{d+1})^{\wed\gamma},
\]
with
\[
\alpha=[\,d-n\in S_m\,],\qquad \beta=2\,[\,d-n+1\in S_m\,],\qquad \gamma=[\,d-n+2\in S_m\,].
\]
\end{enumerate}
\end{lemma}

\begin{proof}
(i) The window index of $L_{m,n-1}$ at quantum grading $q-1$ is
\[
\frac{(q-1)-m-(n-1)}{2}\ =\ \frac{q-m-n}{2}\ =\ d,
\]
the window index of $L_{m,n}$ at $q$. Now apply \Cref{lem:window} to $L_{m,n-1}$ at $q-1$.

(ii) By \eqref{eq:deloop}, applied in quantum grading $q-3n+1$,
\[
A\ =\ \Sigma^{n-1}\Bigl[\,\Xo^{\,q-3n+2}\bigl(T(2,m)\bigr)\ \wed\
\Xo^{\,q-3n}\bigl(T(2,m)\bigr)\Bigr].
\]
By \Cref{lem:factorspectra}, for every $Q$ the spectrum $\Xo^{\,Q}\bigl(T(2,m)\bigr)$ is a
wedge of spheres carrying one sphere $\Sm^{\,i}$ for each $i\in S_m$ with
$Q\in\{m+2i-2,\,m+2i\}$, and no other cells; so $A$ is a wedge of spheres. Writing
$q-m=2d+n$, the two conditions $Q=m+2i-2$ and $Q=m+2i$ read $i=\tfrac{Q-m+2}{2}$ and
$i=\tfrac{Q-m}{2}$, whence
\[
\Xo^{\,q-3n+2}\bigl(T(2,m)\bigr)\ \text{has spheres in dimensions } d-n+1,\ d-n+2,
\]
\[
\Xo^{\,q-3n}\bigl(T(2,m)\bigr)\ \text{has spheres in dimensions } d-n,\ d-n+1,
\]
each present precisely when the dimension in question lies in $S_m$. Applying
$\Sigma^{n-1}$ raises all four dimensions by $n-1$, giving $d,d+1$ from the first summand
and $d-1,d$ from the second. Dimension $d-1$ is therefore reached only by the second
summand and dimension $d+1$ only by the first, while dimension $d$ is reached once by each;
this is the stated count.
\end{proof}

\begin{remark}\label{rem:gammaq}
When does the top sphere of $A$ disappear? Since $S_m=\{0\}\cup\{2,\dots,m\}$, an integer
$t$ misses $S_m$ exactly when $t<0$, $t=1$ or $t>m$. Taking $t=d-n+2$ and using
$q=m+n+2d$, this reads
\[
\gamma=0\quad\Longleftrightarrow\quad q\le m+3n-6,\quad q=m+3n-2,\quad\text{or}\quad
q\ge 3m+3n-2 .
\]
So $\gamma=0$ when the degree $d-n+2$ falls below the support $S_m$ of $\Kho\bigl(T(2,m)\bigr)$,
on its gap at $1$, or above its top degree $m$.

For $A$ itself to vanish, none of $d-n$, $d-n+1$, $d-n+2$ may lie in $S_m$. Three
consecutive integers can all miss $S_m$ only if they are all negative or all larger than
$m$, because three consecutive integers containing $1$ also contain $0$ or $2$. Hence
\[
A\simeq\ast\quad\Longleftrightarrow\quad q\le m+3n-6\ \ \text{or}\ \ q\ge 3m+3n+2 .
\]
The second case plays no role, since $\Xo^{q}(L_{m,n})\simeq\ast$ already for $q>3(m+n)$
(\Cref{rem:fivegradings}).
\end{remark}

\subsection{The connecting map is null or Hopf}\label{ss:dichotomy}

We show that the connecting map $\delta$ of \eqref{eq:skeinLmn} is null-homotopic except
possibly on the top sphere $\Sm^{d+1}$ of $A$, where it factors through the summands
$\Sm^{d}$ of $T$, each component being $0$ or the Hopf map $\eta$. Hence a skein step creates at most
one further copy of $\Sigma^{d-2}\CP^{2}$, and none if $\gamma=0$. The argument is the
inductive step in the proof of \Cref{prop:classification} applied to $\delta$; the decisive
input is again \Cref{lem:etacofiber}.

\begin{lemma}\label{lem:hurewicz}
Let $r\ge2$, let $T$ be an $(r-1)$-connected CW complex with $H_{r}(T;\Z)$ free, let
$A=\bigvee_{i\in I}\Sm^{\,d_{i}}$ be a finite wedge of spheres with $d_{i}\le r$ for every
$i$, and let $\delta\colon A\to T$ be a map with $\delta^{*}=0$ on $H^{r}(-;\Z)$. Then
$\delta\simeq0$.
\end{lemma}

\begin{proof}
A map out of a wedge is null-homotopic if and only if its restriction $\delta_{i}$ to each
summand $\Sm^{d_{i}}$ is. If $d_{i}<r$, then $\delta_{i}\simeq0$ because $T$ is
$(r-1)$-connected. If $d_{i}=r$, the Hurewicz map $\pi_{r}(T)\to H_{r}(T;\Z)$,
$[g]\mapsto g_{*}[\Sm^{r}]$, is an isomorphism \cite[Thm.~4.32]{Hatcher}, so it suffices to
show that $\delta_{*}=0$ on $H_{r}$. By the universal coefficient theorem every homomorphism
$H_{r}(T;\Z)\to\Z$ is evaluation on a class $c\in H^{r}(T;\Z)$, and it vanishes on the image
of $\delta_{*}$ because $\delta^{*}c=0$. Since $H_{r}(T;\Z)$ is free, this forces
$\delta_{*}=0$.
\end{proof}
\begin{proposition}\label{prop:dichotomy}
Let $m,n\ge2$, let $A$ and $T$ be as in
\Cref{lem:source}, and let $\delta\colon A\to T$ be the connecting map of
\eqref{eq:skeinLmn}. Write
\[
T\ \simeq\ (\Sigma^{d-2}\CP^{2})^{\wed\nu'}\wed(\Sm^{d})^{\wed a'}\wed
(\Sm^{d+1})^{\wed b'}\wed(\Sm^{d+2})^{\wed c'}
\]
as in \Cref{cor:classification}, so that by \Cref{lem:source}(ii)
\[
T\wed\Sigma A\ \simeq\ (\Sigma^{d-2}\CP^{2})^{\wed\nu'}\wed(\Sm^{d})^{\wed a'+\alpha}\wed
(\Sm^{d+1})^{\wed b'+\beta}\wed(\Sm^{d+2})^{\wed c'+\gamma}.
\]
\begin{enumerate}[label=\textup{(\roman*)}]
\item If $\delta\simeq0$, then $\Xo^{q}(L_{m,n})\simeq T\wed\Sigma A$.
\item If $\delta\not\simeq0$, then $\gamma=1$ and $a'\ge1$, and $\Xo^{q}(L_{m,n})$ is
obtained from $T\wed\Sigma A$ by replacing one summand $\Sm^{d}\wed\Sm^{d+2}$ with
$\Sigma^{d-2}\CP^{2}$:
\[
\Xo^{q}(L_{m,n})\ \simeq\ (\Sigma^{d-2}\CP^{2})^{\wed\nu'+1}\wed(\Sm^{d})^{\wed a'+\alpha-1}
\wed(\Sm^{d+1})^{\wed b'+\beta}\wed(\Sm^{d+2})^{\wed c'}.
\]
\end{enumerate}
Hence $\nu(m,n;q)$ is $\nu'$ or $\nu'+1$ according as $\delta$ is null or not, where
$\nu'=\nu(m,n-1;q-1)$ if $n\ge3$.
\end{proposition}

\begin{proof}
Part (i) is clear, since the cofiber of a null map $A\to T$ is $T\wed\Sigma A$. For (ii),
suspend \eqref{eq:skeinLmn} so that $d\ge3$, as in the proof of \Cref{cor:classification}.
By \Cref{lem:source}(i) and \Cref{lem:freeduality}, $T$ has free homology in degrees $d,d+1,d+2$, so it is
$(d-1)$-connected, and $\delta^{*}=0$ by \Cref{lem:deltastar}.

Write $A=A'\wed(\Sm^{d+1})^{\wed\gamma}$ with $A'=(\Sm^{d-1})^{\wed\alpha}\wed(\Sm^{d})^{\wed\beta}$.
Since $A'$ has dimension at most $d$, \Cref{lem:hurewicz} gives $\delta|_{A'}\simeq0$. Since
$\gamma\le1$ by \Cref{lem:source}(ii), $\delta\not\simeq0$ forces $\gamma=1$, and $\delta$ is determined by its restriction
$\delta_{1}\in\pi_{d+1}(T)$ to the summand $\Sm^{d+1}$. As in the proof of
\Cref{prop:classification}, only the summands $\Sm^{d}$ and $\Sm^{d+1}$ of $T$ contribute to
$\pi_{d+1}(T)\cong(\Z/2)^{a'}\oplus\Z^{b'}$. The $\Z^{b'}$-component of $\delta_{1}$ consists of
degrees, read off by $\delta^{*}$ in degree $d+1$, so it vanishes by
\Cref{lem:deltastar,lem:hurewicz}. Hence $\delta_{1}$ is a
non-zero map $\bar\delta_{1}\colon\Sm^{d+1}\to(\Sm^{d})^{\wed a'}$ into one summand of $T$; in
particular $a'\ge1$, and
\[
\Xo^{q}(L_{m,n})\ \simeq\ \Cone(\bar\delta_{1})\wed(\Sigma^{d-2}\CP^{2})^{\wed\nu'}\wed
(\Sm^{d+1})^{\wed b'}\wed(\Sm^{d+2})^{\wed c'}\wed\Sigma A' .
\]
By \Cref{lem:etacofiber}, $\Cone(\bar\delta_{1})\simeq\Sigma^{d-2}\CP^{2}\wed(\Sm^{d})^{\wed a'-1}$,
which gives (ii). The count of summands $\Sigma^{d-2}\CP^{2}$ is the uniqueness clause of
\Cref{cor:classification}.
\end{proof}

For $\gamma=0$ this gives the following.

\begin{corollary}\label{cor:deltanull}
Let $m,n\ge2$, and let $A$ and $T$ be as in \Cref{lem:source}. If $\gamma=0$, that is, if
$d-n+2\notin S_{m}$ (by \Cref{rem:gammaq}, if $d<n-2$, $d=n-1$ or $d>m+n-2$), then the
connecting map $\delta$ of \eqref{eq:skeinLmn} is null-homotopic and
\[
\Xo^{\,q}(L_{m,n})\ \simeq\ \Xo^{\,q-1}(L_{m,n-1})\ \wed\ \Sigma A .
\]
If moreover $n\ge3$, then $\nu(m,n;q)=\nu(m,n-1;q-1)$.
\end{corollary}

\begin{remark}\label{rem:residual}
The map $\bar\delta_{1}$ in the proof of \Cref{prop:dichotomy} is invisible to
\Cref{lem:deltastar}, since $\eta$ induces zero on integral cohomology. By
\Cref{lem:etacofiber} it is detected by $\Sq^{2}$ on $\Cone(\bar\delta_{1})$, so deciding whether
$\bar\delta_{1}\simeq0$ is equivalent to computing $\nu$. The arguments here do not do this,
which is why the skein reduction stops at the terminal gradings $\mathcal C(m,n)$ of \Cref{cor:trichotomy}.
\end{remark}
\section{Skein reduction: proof of Theorem B}\label{sec:reduction}

Throughout this section we assume that $q$ is not one of the five gradings of \Cref{thm:vanishing} and that $\Kho^{*,q}(L_{m,n};\Z)\ne0$, that is, $-2\le d\le m+n$; equivalently
\begin{equation}\label{eq:room}
0\le d\le m+n-2\quad\text{and}\quad d\neq1 .\tag{R}
\end{equation}

\begin{definition}\label{def:rho}
For $d\ge0$ with $d\ne1$ define $\rho_d\colon\Z_{\ge2}\to\Z_{\ge2}$ by
\[
\rho_d(k)=\begin{cases}k,& k\le d,\\ d,& k=d+1,\\ d+2,& k\ge d+2.\end{cases}
\]
\end{definition}

\begin{lemma}\label{lem:rho-vs-MN}
$M(q)=\rho_d(m)$ and $N(q)=\rho_d(n)$.
\end{lemma}

\begin{proof}
Immediate from $m\le d\iff q\ge3m+n$, $m=d+1\iff q=3m+n-2$ and $m\ge d+2\iff q\le3m+n-4$.
\end{proof}

\begingroup
\renewcommand{\thetheorem}{B}\renewcommand{\theHtheorem}{thmB}
\begin{theorem}\label{thm:reduction}
Assume \eqref{eq:room}. Then
\[
\Xo^{q}\bigl(L_{m,n}\bigr)\ \simeq\ \Xo^{\,q^{*}}\bigl(L_{M(q),N(q)}\bigr)\ \wed\ W
\]
with $W$ a wedge of spheres, where $q^{*}=M(q)+N(q)+2d$. In particular
$\nu(m,n;q)=\nu\bigl(M(q),N(q);q^{*}\bigr)$.
\end{theorem}
\addtocounter{theorem}{-1}\endgroup

\begin{proof}
Put $M:=\rho_d(m)$ and $N:=\rho_d(n)$; by \Cref{lem:rho-vs-MN} these are $M(q)$ and $N(q)$.
Since $d\ne1$, the definition gives, for all $k\ge2$,
\[
2\le\rho_d(k)\le k,\qquad \rho_d(j)=\rho_d(k)\ \text{ if }\ \rho_d(k)\le j\le k .
\]
Also $M+N\ge d+2$: if $m,n\le d$ this is \eqref{eq:room}; otherwise one of $M,N$ is at
least $d$ and the other at least $2$. So every pair $(m',n')$ with $m'\ge M$ and $n'\ge N$
satisfies \eqref{eq:room}, with the same $d$.

For such a pair, since $d-n'+2\le m'$, the hypothesis $d-n'+2\notin S_{m'}$ of
\Cref{cor:deltanull} is equivalent to $n'=d+1$ or $n'\ge d+3$, that is, to
$\rho_d(n')<n'$. For $N<n'\le n$ we have $\rho_d(n')=N<n'$ and $n'\ge3$, so
\Cref{cor:deltanull} applies at each step from $n$ down to $N$; each step lowers $q$ by one
and splits off a wedge of spheres (\Cref{lem:source}). Hence
\[
\Xo^{q}(L_{m,n})\simeq\Xo^{\,q-(n-N)}(L_{m,N})\wed W_1 .
\]
Since $L_{m,N}=L_{N,m}$, the same argument lowers the first index from $m$ to $M$:
\[
\Xo^{\,q-(n-N)}(L_{m,N})\simeq\Xo^{\,q-(n-N)-(m-M)}(L_{M,N})\wed W_2 .
\]
The final grading is $M+N+2d=q^{*}$, and $W:=W_1\wed W_2$ is a wedge of spheres. Since
$\nu$ is the rank of $\Sq^{2}$ (\Cref{cor:classification}), which vanishes on $W$,
$\nu(m,n;q)=\nu(M,N;q^{*})$.
\end{proof}

\begin{proposition}\label{prop:fixedpoints}
Assume \eqref{eq:room}. Then $(M(q),N(q))=(m,n)$ if and only if $m,n\in\{2,\dots,d\}\cup\{d+2\}$; equivalently, if and only if
\[
\bigl(q\ge3m+n\ \text{ or }\ q=3m+n-4\bigr)\quad\text{and}\quad\bigl(q\ge m+3n\ \text{ or }\ q=m+3n-4\bigr).
\]
\end{proposition}

\begin{proof}
$\rho_d(k)=k$ iff $k\le d$ or $k=d+2$; translate using \Cref{lem:rho-vs-MN}.
\end{proof}

Theorems A and B, together with \Cref{prop:fixedpoints}, sort the quantum gradings in
which $\Xo^{q}(L_{m,n})$ is not contractible (\Cref{rem:fivegradings}) into three kinds.

\begin{corollary}\label{cor:trichotomy}
Let $m,n\ge2$. The quantum gradings $q=m+n+2d$, $-2\le d\le m+n$, of $L_{m,n}$ fall into
three disjoint sets:
\begin{itemize}
\item $\mathcal A(m,n)$, the five gradings with $d\in\{-2,-1,1,m+n-1,m+n\}$ of
\Cref{thm:vanishing};
\item $\mathcal C(m,n)$, the gradings with $d\ne1$, $d\le m+n-2$ and
$m,n\in\{2,\dots,d\}\cup\{d+2\}$ of \Cref{prop:fixedpoints};
\item $\mathcal B(m,n)$, the remaining ones, namely those with $0\le d\le m+n-2$ and
$d\ne1$ in which $m$ or $n$ equals $d+1$ or is at least $d+3$; these are the gradings in
which \Cref{thm:reduction} strictly lowers $(m,n)$.
\end{itemize}
If $q\in\mathcal B(m,n)$, then the grading $q^{*}$ of \Cref{thm:reduction} lies in
$\mathcal C\bigl(M(q),N(q)\bigr)$.
\end{corollary}

\begin{proof}
A value of $d$ in $[-2,m+n]$ other than the five values of $\mathcal A(m,n)$ satisfies
\eqref{eq:room}, and the grading $m+n+2d$ then lies in $\mathcal C(m,n)$ or in
$\mathcal B(m,n)$ according as $m,n\in\{2,\dots,d\}\cup\{d+2\}$ or not. By definition $\mathcal{B}$ and $\mathcal{C}$ are disjoint.

Let $q\in\mathcal B(m,n)$ and put $M:=M(q)$, $N:=N(q)$. The first paragraph of the proof
of \Cref{thm:reduction} shows that $M+N\ge d+2$; taking $j=\rho_d(k)$ in the second property
displayed there gives $\rho_d(\rho_d(k))=\rho_d(k)$, so $\rho_{d}(M)=M$ and $\rho_{d}(N)=N$,
that is, $M,N\in\{2,\dots,d\}\cup\{d+2\}$. Since $d(q^{*})=d$, this says that
$q^{*}\in\mathcal C(M,N)$.
\end{proof}

On $\mathcal C(m,n)$ neither \Cref{thm:vanishing} nor \Cref{thm:reduction} determines $\nu$,
so these are the gradings that have to be treated directly. Whether $m+n+2d$ lies in
$\mathcal C(m,n)$ depends only on $d$ and on the unordered pair $\{m,n\}$. We therefore fix
$d$ and collect the pairs for which it does.

\begin{definition}\label{def:Id}
For $d\ge0$ with $d\ne1$, let $I_d$ be the set of unordered pairs $\{m',n'\}$ with
\[
m',n'\in\{2,\dots,d\}\cup\{d+2\}\qquad\text{and}\qquad m'+n'\ge d+2 .
\]
\end{definition}

Thus $m+n+2d\in\mathcal C(m,n)$ if and only if $\{m,n\}\in I_d$, and by
\Cref{cor:trichotomy} the skein reduction at level $d$ always ends in $I_d$. For a fixed $d$ there
are infinitely many pairs $\{m,n\}$ with a grading at level $d$, but, as we now show, only
finitely many of them satisfy this condition.

\begin{corollary}\label{cor:finite}
Let $d\ge0$ with $d\ne1$. Then $I_d$ is finite: $|I_0|=1$, and
$|I_d|=d+\lfloor d^{2}/4\rfloor$ for $d\ge2$, so that $|I_2|=3$, $|I_3|=5$, $|I_4|=8$ and
$|I_5|=11$.
\end{corollary}

\begin{proof}
Put $K_d:=\{2,\dots,d\}\cup\{d+2\}$. For $d=0$ we have $K_0=\{2\}$, so
$I_0=\bigl\{\{2,2\}\bigr\}$. Let $d\ge2$, so that $|K_d|=d$. Every pair $\{a,d+2\}$ with
$a\in K_d$ lies in $I_d$, which gives $d$ pairs. The other members of $I_d$ are the pairs
$\{a,b\}$ with $2\le a\le b\le d$ and $a+b\ge d+2$. For fixed $a$ these are the $b$ with
$\max(a,d+2-a)\le b\le d$, and there are $\min(a-1,d+1-a)$ of them. Summing over
$a=2,\dots,d$ and writing $t=a-1$, their number is
\[
\sum_{t=1}^{d-1}\min(t,d-t)\ =\ \lfloor d^{2}/4\rfloor ,
\]
since for $d=2e$ the sum is $2(1+\dots+(e-1))+e=e^{2}$, and for $d=2e+1$ it is
$2(1+\dots+e)=e^{2}+e$. Hence $|I_d|=d+\lfloor d^{2}/4\rfloor$.
\end{proof}

\begin{figure}[t]
\centering
\begin{tikzpicture}[scale=0.62,>=stealth]
\foreach \m in {2,3,4,6}{\foreach \n in {2,3,4,6}{
  \pgfmathtruncatemacro{\mnsum}{\m+\n}
  \ifnum\mnsum>5 \fill[black!14] (\m-.46,\n-.46) rectangle (\m+.46,\n+.46);\fi}}
\foreach \n in {5,7,8,9}{\foreach \m in {2,...,9}{
  \draw[->,black!40] (\m,\n-.2) -- (\m,\n-.8);}}
\foreach \m in {5,7,8,9}{\foreach \n in {2,...,9}{
  \draw[->,black!40] (\m-.2,\n) -- (\m-.8,\n);}}
\foreach \n in {9,8,7}{\draw[->,thick] (9,\n-.2) -- (9,\n-.8);}
\foreach \m in {9,8,7}{\draw[->,thick] (\m-.2,6) -- (\m-.8,6);}
\foreach \m in {2,...,9}{\foreach \n in {2,...,9}{
  \pgfmathtruncatemacro{\mnsum}{\m+\n}
  \ifnum\mnsum>5 \fill (\m,\n) circle (2.1pt);\else \draw (\m,\n) circle (2.1pt);\fi}}
\draw[thick] (6,6) circle (6.5pt);
\foreach \k in {2,...,9}{
  \node[font=\scriptsize] at (\k,1.25) {$\k$};
  \node[font=\scriptsize] at (1.25,\k) {$\k$};}
\node[font=\scriptsize] at (4,0.75) {$d$};
\node[font=\scriptsize] at (5,0.75) {$d{+}1$};
\node[font=\scriptsize] at (6,0.75) {$d{+}2$};
\node[font=\scriptsize,anchor=east] at (0.95,4) {$d$};
\node[font=\scriptsize,anchor=east] at (0.95,5) {$d{+}1$};
\node[font=\scriptsize,anchor=east] at (0.95,6) {$d{+}2$};
\node at (9.9,1.25) {$m$};
\node at (1.25,9.9) {$n$};
\draw (11.0,8.6) circle (2.1pt);
\node[font=\scriptsize,anchor=west] at (11.3,8.6) {$\mathcal A$};
\fill (11.0,7.9) circle (2.1pt);
\node[font=\scriptsize,anchor=west] at (11.3,7.9) {$\mathcal B$};
\fill[black!14] (11.0-.3,7.2-.3) rectangle (11.0+.3,7.2+.3);
\fill (11.0,7.2) circle (2.1pt);
\node[font=\scriptsize,anchor=west] at (11.3,7.2) {$\mathcal C$};
\end{tikzpicture}
\caption{The skein reduction at level $d=4$, for $2\le m,n\le9$; see \Cref{rem:diagonal}.}
\label{fig:reduction}
\end{figure}
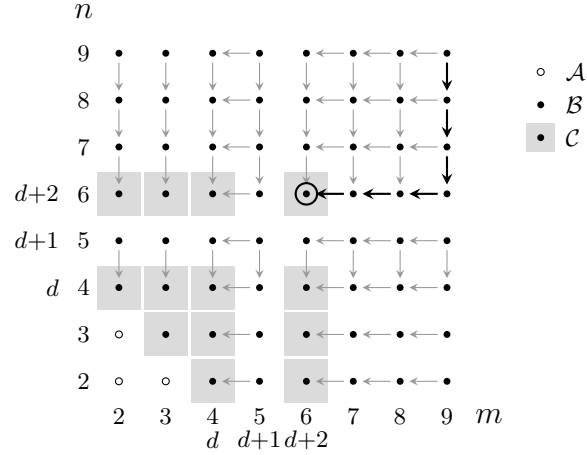

\begin{remark}[Where the skein reduction terminates from above]\label{rem:diagonal}
Fix $d$ and read $\gamma=[\,d-n+2\in S_m\,]$ as $n$ decreases. For $n\ge d+3$ the source
$A$ vanishes (\Cref{rem:gammaq}), so each step is an equivalence and
\[
\Xo^{\,m+n+2d}(L_{m,n})\ \simeq\ \Xo^{\,m+3d+2}(L_{m,d+2})\qquad\text{for all }m\ge2,\ n\ge d+2 .
\]
The first obstruction occurs at $n=d+2$, where $d-n+2=0\in S_m$ for every $m$. Since
$\rho_d(k)=d+2$ for $k\ge d+2$, \Cref{thm:reduction} reduces every pair $(m,n)$ with $n\ge d+2$
to $(\rho_d(m),d+2)$, and every pair with $m\ge d+2$ to $(d+2,\rho_d(n))$; if $m,n\ge d+2$,
both apply, and the skein reduction terminates at the diagonal pair $(d+2,d+2)$, in quantum
grading $4d+4$. Since $\rho_d^{-1}(d+2)=\{d+2,d+3,\dots\}$, while $\rho_d^{-1}(k)$ has at most
two elements for $k\le d$, the members of $I_d$ reached from infinitely many pairs are exactly
the pairs $\{M,d+2\}$ with $M\in\{2,\dots,d\}\cup\{d+2\}$. These are the pairs $(M,d+2)$ whose value of $\nu$ appears in \Cref{thm:onesided}, and in
\Cref{thm:stabilization} for $M=d+2$.

\Cref{fig:reduction} shows the case $d=4$. The point $(m,n)$ stands for
$\Xo^{\,m+n+8}(L_{m,n})$ and is drawn as an open circle, a dot or a shaded dot according as
its grading lies in $\mathcal A(m,n)$, $\mathcal B(m,n)$ or $\mathcal C(m,n)$
(\Cref{cor:trichotomy}). Thus the shaded pairs are those with $\{m,n\}\in I_4$, where
\[
I_4=\bigl\{\{2,4\},\{2,6\},\{3,3\},\{3,4\},\{3,6\},\{4,4\},\{4,6\},\{6,6\}\bigr\};
\]
the five pairs with distinct entries appear twice, which gives $13$ shaded points. An arrow
$(m,n)\to(m,n-1)$ or $(m,n)\to(m-1,n)$ is a skein step to which \Cref{cor:deltanull}
applies, so $\nu$ is the same at both ends; every pair in $\mathcal B$ has such an arrow, and
no shaded pair has one. The black path is the reduction from $(9,9)$ in the proof of
\Cref{thm:reduction}, which first lowers $n$ and then $m$. In the row $n=6$ the pair $(5,6)$
moves left to $(4,6)$, and the pairs $(m,6)$ with $m\ge7$ move left to $(6,6)$; so $(4,6)$ is
the end of the columns $m=4$ and $m=5$ above that row, and the circled pair $(6,6)$ is the
end of the whole quadrant $m,n\ge6$. The shaded pairs with $m,n\le4$, by contrast, are
reached only from pairs with $m,n\le5$.
\end{remark}

\section{Stabilization: proof of Theorem C}\label{sec:stabilization}

\begingroup
\renewcommand{\thetheorem}{C}\renewcommand{\theHtheorem}{thmC}
\begin{theorem}\label{thm:stabilization}
Fix $d\ge3$ and set $\nu_d:=\nu(d+2,d+2;\,4d+4)$. Then for all $m,n\ge d+2$,
\[
\Xo^{\,m+n+2d}\bigl(L_{m,n}\bigr)\ \simeq\ (\Sigma^{d-2}\CP^{2})^{\wed\nu_d}\ \wed\ (\Sm^{d})^{\wed\,d-1-\nu_d}\ \wed\ (\Sm^{d+1})^{\wed\,2d}\ \wed\ (\Sm^{d+2})^{\wed\,d+1-\nu_d}.
\]
\end{theorem}
\addtocounter{theorem}{-1}\endgroup

\begin{proof}
Since $m,n\ge d+2$, \Cref{lem:rho-vs-MN} gives $(M,N)=(d+2,d+2)$ and $q^{*}=4d+4$, so $\nu(m,n;m+n+2d)=\nu_d$ by \Cref{thm:reduction}. For the ranks, $m,n\ge d+2>d$ so \Cref{lem:Pvanishing} gives $P(d)=d-1$, $P(d+1)=d$ and $P(d+2)=d+1$ for $d\ge3$. Substituting into \eqref{eq:intro-classification} gives the statement.
\end{proof}

\Cref{thm:stabilization} is the case $m\ge d+2$ of the following stabilization in the second
index alone, which makes the display in \Cref{rem:diagonal} explicit.

\begin{theorem}\label{thm:onesided}
Fix $d\ge3$ and $m\ge2$, and put $M:=\rho_d(m)$. For every $n\ge d+2$,
\[
\Xo^{\,m+n+2d}(L_{m,n})\ \simeq\ (\Sigma^{d-2}\CP^{2})^{\wed\nu}\wed(\Sm^{d})^{\wed\,a-\nu}
\wed(\Sm^{d+1})^{\wed\,b}\wed(\Sm^{d+2})^{\wed\,c-\nu},
\]
where $\nu=\nu(M,d+2;M+3d+2)$ and
$(a,b,c)=\bigl(P(d),2P(d+1),P(d+2)\bigr)$.
\end{theorem}

\begin{proof}
Let $n\ge d+2$ and put $q:=m+n+2d$, so that $d(q)=d$. Since $d\ge3$ and
$m+n-2\ge n\ge d+2$, the grading $q$ satisfies \eqref{eq:room}. By \Cref{lem:rho-vs-MN},
$M(q)=\rho_d(m)=M$ and $N(q)=\rho_d(n)=d+2$, so $q^{*}=M+(d+2)+2d=M+3d+2$, and
\Cref{thm:reduction} gives
\[
\nu(m,n;q)\ =\ \nu(M,d+2;\,M+3d+2)\ =\ \nu .
\]
Substituting this value of $\nu$ and $(a,b,c)=\bigl(P(d),2P(d+1),P(d+2)\bigr)$ into
\eqref{eq:intro-classification} gives the displayed equivalence.
\end{proof}

\section{Examples: the propagation of \texorpdfstring{$q=14$}{q=14}}\label{sec:examples}

We follow the summand $\Sigma^{2}\CP^{2}$ of \eqref{eq:schutz-example} from $L_{3,3}$ in
quantum grading $14$ through the family. Since $14=3+3+2\cdot4$, the relevant gradings
are those of window index $4$,
\[
q\ =\ m+n+8,
\]
and throughout this section $d=4$. By \Cref{lem:window} the spectrum
$\Xo^{m+n+8}(L_{m,n})$ has free homology of ranks $\bigl(P(4),2P(5),P(6)\bigr)$ in degrees
$4,5,6$, so by \Cref{cor:classification}
\begin{equation}\label{eq:d4type}
\Xo^{\,m+n+8}(L_{m,n})\ \simeq\ (\Sigma^{2}\CP^{2})^{\wed\nu}\wed(\Sm^{4})^{\wed P(4)-\nu}
\wed(\Sm^{5})^{\wed 2P(5)}\wed(\Sm^{6})^{\wed P(6)-\nu},
\end{equation}
with $\nu=\nu(m,n;m+n+8)$. We show that $\nu=1$ for all $m,n\ge3$. The arguments of
\Cref{sec:skein,sec:reduction} reduce this to the single value $\nu_{4}=\nu(6,6;20)$, and one
KnotJob computation supplies that value.

\subsection{The sources}

When a crossing of the $T(2,n)$ factor is resolved at $d=4$, the source of
\eqref{eq:skeinLmn} does not depend on $m$.

\begin{lemma}\label{lem:sources4}
Let $m\ge2$, $n\ge4$ and $q=m+n+8$, and let $A$ be as in \Cref{lem:source}. Then
\[
A\ \simeq\
\begin{cases}
\Sm^{3}\wed\Sm^{5}, & n=4,\\
\Sm^{4}\wed\Sm^{4}, & n=5,\\
\Sm^{5}, & n=6,\\
\ast, & n\ge7 .
\end{cases}
\]
In particular $\gamma=1$ if $n\in\{4,6\}$, and $\gamma=0$ if $n=5$ or $n\ge7$.
\end{lemma}

\begin{proof}
By \Cref{lem:source}(ii) with $d=4$,
$A\simeq(\Sm^{3})^{\wed\alpha}\wed(\Sm^{4})^{\wed\beta}\wed(\Sm^{5})^{\wed\gamma}$ with
$\alpha=[\,4-n\in S_m\,]$, $\beta=2\,[\,5-n\in S_m\,]$ and $\gamma=[\,6-n\in S_m\,]$. As
$m\ge2$ we have $0,2\in S_m$ and $1\notin S_m$, and no negative integer lies in $S_m$. So
$(\alpha,\beta,\gamma)$ is $(1,0,1)$ for $n=4$, $(0,2,0)$ for $n=5$, $(0,0,1)$ for $n=6$,
and $(0,0,0)$ for $n\ge7$.
\end{proof}

\subsection{Three kinds of step}\label{ss:steps}

Consider a step $L_{m,n-1}\rightsquigarrow L_{m,n}$ with $m\ge2$ and $n\ge4$, write
$T=\Xo^{\,q-1}(L_{m,n-1})$, and let $a'$ be the number of summands $\Sm^{4}$ of $T$ in
\eqref{eq:d4type}, so that $a'=P_{m,n-1}(4)-\nu(m,n-1;q-1)$. By \Cref{lem:sources4} each
step is of one of the following kinds.
\begin{enumerate}[label=\textup{(\roman*)}]
\item If $n=5$ or $n\ge7$, then $\gamma=0$ and \Cref{cor:deltanull} gives
$\Xo^{q}(L_{m,n})\simeq T\wed\Sigma A$, with $\nu$ unchanged. For $n\ge7$ the source is
contractible and the step is an equivalence $\Xo^{q}(L_{m,n})\simeq T$.
\item If $n\in\{4,6\}$, then $\gamma=1$, and by \Cref{prop:dichotomy} either $\delta\simeq0$
and $\nu$ is unchanged, or $\delta\not\simeq0$ and $\nu$ increases by exactly one.
\item If in case (ii) we have $a'=0$, there is nowhere for a Hopf map to land:
$\delta\simeq0$ by \Cref{prop:dichotomy}(ii), and $\nu$ is unchanged.
\end{enumerate}
For the first step, $L_{3,3}\rightsquigarrow L_{3,4}$ in grading $15$, we have $n=4$, so
$\gamma=1$ and $A\simeq\Sm^{3}\wed\Sm^{5}$ by \Cref{lem:sources4}. By
\eqref{eq:schutz-example}, $T=\Xo^{14}(L_{3,3})\simeq\Sigma^{2}\CP^{2}\wed(\Sm^{5})^{\wed4}$
has no summand $\Sm^{4}$, that is, $a'=P_{3,3}(4)-\nu(3,3;14)=1-1=0$. So the step is of
kind (iii): $\delta\simeq0$, and $\Xo^{15}(L_{3,4})\simeq T\wed\Sigma A$.
\subsection{Monotonicity}

\begin{proposition}\label{prop:propagation}
For all $m,n\ge3$,
\[
1\ \le\ \nu(m,n;m+n+8)\ \le\ \nu_{4}.
\]
In particular $\nu_{4}\ge1$. If $\nu_{4}=1$, then $\nu(m,n;m+n+8)=1$ for all $m,n\ge3$, and
the connecting map of every step $L_{m,n-1}\rightsquigarrow L_{m,n}$ with
$3\le m\le6$ and $4\le n\le6$ is null-homotopic.
\end{proposition}

\begin{proof}
By \Cref{prop:dichotomy}, $\nu(m,n;q)\ge\nu(m,n-1;q-1)$ whenever $m\ge2$ and $n\ge3$, and by
$L_{m,n}=L_{n,m}$ the same holds in the first index. Raising the indices one at a time
from $(3,3)$ therefore gives $\nu(m,n;m+n+8)\ge\nu(3,3;14)=1$, the equality being
\eqref{eq:schutz-example}. If $3\le m,n\le6$, raising them one at a time from $(m,n)$ to
$(6,6)$ gives $\nu(m,n;m+n+8)\le\nu(6,6;20)=\nu_{4}$. For general $m,n\ge3$ the condition
\eqref{eq:room} holds, since $4\le m+n-2$, so \Cref{thm:reduction} gives
$\nu(m,n;m+n+8)=\nu\bigl(\rho_{4}(m),\rho_{4}(n);\rho_{4}(m)+\rho_{4}(n)+8\bigr)$; and
$\rho_{4}(k)\in\{3,4,6\}$ for every $k\ge3$, which returns us to the previous case.

If $\nu_{4}=1$ the two bounds agree. A step with $3\le m\le6$ and $4\le n\le6$ lies on a
chain from $(3,3)$ to $(6,6)$ along which $\nu$ is constant, so $\nu$ does not increase
at that step, and $\delta\simeq0$ by \Cref{prop:dichotomy}.
\end{proof}

\subsection{The computation up to \texorpdfstring{$T(2,6)\sqcup T(2,6)$}{T(2,6) + T(2,6)}}\label{ss:d4computation}

\Cref{tab:d4} lists, for $3\le m\le n\le6$, one step reaching $L_{m,n}$, its source from
\Cref{lem:sources4}, the kind of step from \Cref{ss:steps}, and the resulting spectrum.
Steps in the first index are read through $L_{m,n}=L_{n,m}$.

\begin{table}[ht]
\centering
\small
\begin{tabular}{cccccc}
\hline
$(m,n)$ & $q$ & from & $A$ & kind & $(a,b,c)$\\
\hline
$(3,3)$ & $14$ & --- & --- & \eqref{eq:schutz-example} & $(0,4,0)$\\
$(3,4)$ & $15$ & $(3,3)$ & $\Sm^{3}\wed\Sm^{5}$ & (iii), $a'=0$ & $(1,4,1)$\\
$(3,5)$ & $16$ & $(3,4)$ & $\Sm^{4}\wed\Sm^{4}$ & (i) & $(1,6,1)$\\
$(3,6)$ & $17$ & $(3,5)$ & $\Sm^{5}$ & (ii), $a'=1$ & $(1,6,2)$\\
$(4,4)$ & $16$ & $(3,4)$ & $\Sm^{3}\wed\Sm^{5}$ & (ii), $a'=1$ & $(2,4,2)$\\
$(4,5)$ & $17$ & $(4,4)$ & $\Sm^{4}\wed\Sm^{4}$ & (i) & $(2,6,2)$\\
$(4,6)$ & $18$ & $(4,5)$ & $\Sm^{5}$ & (ii), $a'=2$ & $(2,6,3)$\\
$(5,5)$ & $18$ & $(4,5)$ & $\Sm^{4}\wed\Sm^{4}$ & (i) & $(2,8,2)$\\
$(5,6)$ & $19$ & $(4,6)$ & $\Sm^{4}\wed\Sm^{4}$ & (i) & $(2,8,3)$\\
$(6,6)$ & $20$ & $(5,6)$ & $\Sm^{5}$ & (ii), $a'=2$ & $(2,8,4)$\\
\hline
\end{tabular}
\medskip
\caption{Window index $d=4$. In every row
$\Xo^{q}(L_{m,n})\simeq\Sigma^{2}\CP^{2}\wed(\Sm^{4})^{\wed a}\wed(\Sm^{5})^{\wed b}\wed
(\Sm^{6})^{\wed c}$, where $(a+1,b,c+1)=\bigl(P(4),2P(5),P(6)\bigr)$.}
\label{tab:d4}
\end{table}

The first three rows are proved outright: $(3,3)$ is \eqref{eq:schutz-example}, $(3,4)$ is
a step of kind (iii), and $(3,5)$ follows by \Cref{cor:deltanull}. In the remaining rows
\Cref{prop:propagation} gives $1\le\nu\le\nu_{4}$, and the steps of kind (ii) with $a'\ge1$
are exactly those at which a Hopf map is not excluded by the theory. A KnotJob computation
\cite{KnotJob} of $\Sq^{2}_{1}$ on $\Kho^{4,20}(L_{6,6};\F_{2})$ gives rank
$1$, so $\nu_{4}=1$ by \Cref{cor:classification} and \Cref{ss:oddspectrum}. By
\Cref{prop:propagation} this single value gives $\nu(m,n;m+n+8)=1$ for all $m,n\ge3$, in
particular in every row of the table, and forces every connecting map in the table to be
null. As consistency checks, KnotJob also gives rank $1$
at $L_{4,4}$, $L_{4,6}$ and $L_{5,6}$ in gradings $16$, $18$ and $19$.

Two features of the table are worth recording. First, \emph{for $m,n<6$ every connecting
map at $d=4$ vanishes}, and here less input suffices: the only steps with $m,n\le5$ at which
a Hopf map has room are $L_{4,3}\rightsquigarrow L_{4,4}$ and
$L_{5,3}\rightsquigarrow L_{5,4}$, and both are settled by the value $\nu(4,4;16)=1$, the
second because $\nu(5,4;17)=\nu(4,4;16)$ by \Cref{cor:deltanull} while
$\nu(5,3;16)=\nu(3,5;16)=1$. Secondly, the entry $(5,6)$ is reached from $(4,6)$ by a step
of kind (i), so $\nu(5,6;19)=\nu(4,6;18)$, while $\nu(5,5;18)=\nu(4,5;17)$ by a step of kind
(i) as well. Hence the step $L_{5,5}\rightsquigarrow L_{5,6}$, which is of kind (ii), has
$\delta\simeq0$ if and only if the step $L_{4,5}\rightsquigarrow L_{4,6}$ has; the theory alone
does not decide which, and $\nu_4=1$ gives $\delta\simeq0$ for both.

\subsection{Stability}

By kind (i) every step with $n\ge7$ is an equivalence, so, as in \Cref{rem:diagonal}, for
each $m\ge2$ the spectra $\Xo^{m+n+8}(L_{m,n})$ with $n\ge6$ are all equivalent to
$\Xo^{m+14}(L_{m,6})$. The rows of \Cref{tab:d4} with second index $6$ therefore give every
larger second index.

\begin{example}\label{ex:onesided}
At $d=4$ and for $m\ge3$, \Cref{thm:onesided} holds with $\nu=1$ (\Cref{ss:d4computation}). For
$m=3,4,5$ the ranks $\bigl(P(4),2P(5),P(6)\bigr)$ are $(2,6,3)$, $(3,6,4)$ and $(3,8,4)$, so
that for all $n\ge6$
\[
\begin{aligned}
\Xo^{\,n+11}(L_{3,n})&\ \simeq\ \Sigma^{2}\CP^{2}\wed\Sm^{4}\wed(\Sm^{5})^{\wed6}
\wed(\Sm^{6})^{\wed2},\\
\Xo^{\,n+12}(L_{4,n})&\ \simeq\ \Sigma^{2}\CP^{2}\wed(\Sm^{4})^{\wed2}\wed(\Sm^{5})^{\wed6}
\wed(\Sm^{6})^{\wed3},\\
\Xo^{\,n+13}(L_{5,n})&\ \simeq\ \Sigma^{2}\CP^{2}\wed(\Sm^{4})^{\wed2}\wed(\Sm^{5})^{\wed8}
\wed(\Sm^{6})^{\wed3}.
\end{aligned}
\]
For $m\ge6$ the ranks are $(3,8,5)$, which gives the following corollary.
\end{example}

\begin{corollary}\label{cor:d4}
For all $m,n\ge6$,
\[
\Xo^{\,m+n+8}\bigl(T(2,m)\sqcup T(2,n)\bigr)\ \simeq\ \Sigma^{2}\CP^{2}\ \wed\ (\Sm^{4})^{\wed2}\ \wed\ (\Sm^{5})^{\wed8}\ \wed\ (\Sm^{6})^{\wed4}.
\]
\end{corollary}

\begin{proof}
Take $d=4$ in \Cref{thm:stabilization}, with $\nu_{4}=1$ (\Cref{ss:d4computation}) and
$\bigl(P(4),2P(5),P(6)\bigr)=(3,8,5)$.
\end{proof}

\subsection{Beyond \texorpdfstring{$d=4$}{d=4}: a Hopf connecting map}

\Cref{tab:kj} records, for the pairs we computed, the gradings $q$ with $\nu(m,n;q)\ne0$ and
the values $\nu(m,n;q)$, which by \Cref{cor:classification} and \Cref{ss:oddspectrum} are
the ranks of $\Sq^{2}_{1}$ computed with KnotJob.
The first entry of each row is $q=m+n+8$, in agreement with \Cref{prop:propagation}; the
later entries have $d=6,8,10$.

\begin{table}[ht]
\centering
\small
\begin{tabular}{cl}
\hline
$(m,n)$ & $q:\nu$\\
\hline
$(3,4)$ & $15:1$\\
$(3,5)$ & $16:1,\ 20:1$\\
$(3,6)$ & $17:1,\ 21:1$\\
$(3,7)$ & $18:1,\ 22:1,\ 26:1$\\
$(3,8)$ & $19:1,\ 23:1,\ 27:1$\\
$(3,9)$ & $20:1,\ 24:1,\ 28:1,\ 32:1$\\
$(4,4)$ & $16:1$\\
$(4,5)$ & $17:1,\ 21:1$\\
$(4,6)$ & $18:1,\ 22:1$\\
$(4,7)$ & $19:1,\ 23:1,\ 27:1$\\
$(4,8)$ & $20:1,\ 24:1,\ 28:1$\\
$(5,5)$ & $18:1,\ 22:2,\ 26:1$\\
$(5,6)$ & $19:1,\ 23:2,\ 27:1$\\
$(5,7)$ & $20:1,\ 24:2,\ 28:2,\ 32:1$\\
$(6,6)$ & $20:1,\ 24:2,\ 28:1$\\
\hline
\end{tabular}
\medskip
\caption{For each pair $(m,n)$, the quantum gradings $q$ with $\nu(m,n;q)\neq0$, written
$q:\nu(m,n;q)$. Here $\nu(m,n;q)$ is the rank of $\Sq^{2}_{1}$ on
$\Kho^{d,q}(L_{m,n};\F_2)$, computed with KnotJob.}
\label{tab:kj}
\end{table}

\begin{remark}\label{rem:hopfd6}
Consider the step $L_{3,4}\rightsquigarrow L_{3,5}$ at $d=6$, from grading $19$ to grading
$20$. For $L_{3,4}$ we have $d=6=m+n-1$, so $T=\Xo^{19}(L_{3,4})$ is a wedge of spheres by
\Cref{thm:vanishing}, namely $(\Sm^{6})^{\wed2}\wed(\Sm^{7})^{\wed2}$. By
\Cref{lem:source}(ii), $A\simeq(\Sm^{6})^{\wed2}\wed\Sm^{7}$, so $\gamma=1$ and $a'=2$, and
\Cref{prop:dichotomy} gives $\nu(3,5;20)=0$ if $\delta\simeq0$ and $\nu(3,5;20)=1$
otherwise. \Cref{tab:kj} lists $q=20$ for $L_{3,5}$, so $\delta\not\simeq0$, and case (ii)
of \Cref{prop:dichotomy}, with $(\nu',a',b',c')=(0,2,2,0)$ and
$(\alpha,\beta,\gamma)=(0,2,1)$, gives
\[
\Xo^{20}(L_{3,5})\ \simeq\ \Sigma^{4}\CP^{2}\wed\Sm^{6}\wed(\Sm^{7})^{\wed4} .
\]
So at $d=6$ a skein step with $m,n<6$ does create a summand $\CP^{2}$: the vanishing of all
such connecting maps at $d=4$ (\Cref{ss:d4computation}) does not persist.
\end{remark}

\section{Closing remarks, open questions and future directions}\label{sec:questions}
By \eqref{eq:intro-classification}, determining $\Xo^{q}(L_{m,n})$ comes down to determining
\begin{equation}\label{eq:nu-Sq21}
\nu(m,n;q)=\rk_{\F_2}\bigl(\Sq^2_1\colon\Kho^{d,q}(L_{m,n};\F_2)
\longrightarrow\Kho^{d+2,q}(L_{m,n};\F_2)\bigr),
\end{equation}
for which we have no general formula. We compare this rank with the even square, then state questions about the
remaining attaching maps.

\subsection{Relation to earlier questions}\label{ss:earlier}
Question~(q-6) of \cite[\S5.7]{SSS} asks whether the (equivariant) odd Khovanov spectrum
of $L_1\sqcup L_2$ can be expressed in terms of the spectra of $L_1$ and $L_2$. For
$L_{m,n}$, after forgetting the $\Z/2$-action, \Cref{thm:reduction} reduces
$\Xo^{q}(L_{m,n})$ to $m$, $n$, $q$ and finitely many values of $\nu$
(\Cref{cor:intro-irreducible}), but not to smash products: every wedge of smash products of
the spectra $\Xo^{q_i}(T(2,k_i))$ is a wedge of spheres (\Cref{lem:factorspectra}), whereas
$\Xo^{m+n+8}(L_{m,n})$ has a summand $\Sigma^{2}\CP^{2}$ for all $m,n\ge3$
(\Cref{prop:propagation}). For $m=n=3$ this was observed by Sch\"utz
\cite[\S5.2]{Schutz2022}, assuming that his squares are Steenrod squares of a spectrum, and
Rajapakse proved that $\X_{k}(T(2,3)\sqcup T(2,3))$ is not the smash product for every odd
$k$ \cite[Thm.~1.3]{Rajapakse}.

\emph{Smash products of Moore spectra.} A smash product of Moore spectra with free integral
homology is a wedge of spheres: write each factor as $M(F)\wed M(T)$ with $F$ free and $T$
finite and distribute; every summand involving some $M(T)$ has finite homology. Since
$\Kho(L_{m,n};\Z)$ is free, any wedge of such smash products equivalent to $\Xo^{q}(L_{m,n})$
would be a wedge of spheres, so by \Cref{prop:propagation} $L_{m,n}$ with $m,n\ge3$ answers \cite[\S5.7, (q-3)]{SSS}.
(Freeness is essential: $\Sigma^{2}(\RP^{2}\wedge\RP^{2})$ is a smash of Moore spectra with
$\Sq^{2}\neq0$.)

\emph{Summands $\CP^{2}$.} In the even theory the gradings $m+n+8$ of
\Cref{prop:propagation} carry instead one summand $\Sigma^{2}(\RP^{2}\wedge\RP^{2})$, whose
integral cohomology is $\Z/2$ in degrees $5$ and $6$ (\Cref{prop:even} below, as
$\mu(m,n;m+n+8)=1$), in accordance with \cite[Conj.~2.1]{LOS}.

\emph{Reduced spectra.} Sarkar, Scaduto and Stoffregen construct, for a basepoint $p$ on
$L$, reduced odd spectra $\widetilde{\X}_{o}^{\,j}(L,p)$ whose cohomology is
$\Khro^{*,j}(L;\Z)$, and a cofibration sequence
\[
\widetilde{\X}_{o}^{\,q-1}(L,p)\longrightarrow\Xo^{q}(L)\longrightarrow
\widetilde{\X}_{o}^{\,q+1}(L,p)
\]
\cite[Thm.~1.2]{SSS}. They ask whether $\Xo^{q}(L)$ is equivalent to the wedge of the two
outer terms, as \Cref{prop:orsplit} suggests \cite[\S5.7, (q-2)]{SSS}. Let $L=L_{m,n}$ and
$q\equiv m+n\pmod 2$. By \Cref{prop:orsplit}, $\Khro^{*,q-1}(L_{m,n};\Z)$ is a direct summand
of $\Kho^{*,q-2}(L_{m,n};\Z)$ and of $\Kho^{*,q}(L_{m,n};\Z)$, so by \Cref{lem:window} it is
free and concentrated in degrees $d,d+1$; likewise $\Khro^{*,q+1}(L_{m,n};\Z)$ is free and
concentrated in degrees $d+1,d+2$. By the proof of \Cref{lem:factorspectra} both reduced
spectra are therefore wedges of spheres, and as their ranks add up to those of
$\Xo^{q}(L_{m,n})$,
\[
\Xo^{q}(L_{m,n})\ \simeq\ \widetilde{\X}_{o}^{\,q-1}(L_{m,n},p)\wed
\widetilde{\X}_{o}^{\,q+1}(L_{m,n},p)\quad\Longleftrightarrow\quad\nu(m,n;q)=0 .
\]
So the answer to \cite[(q-2)]{SSS} is negative for $L_{m,n}$ in grading $m+n+8$ whenever
$m,n\ge3$. A first negative answer, for $T(3,-4)$ and $\X_{1}$, is due to Rajapakse
\cite[Thm.~1.4]{Rajapakse}. (That $\Sq^{2}_{\varepsilon}$ does not split over reduced odd Khovanov homology is
observed for knots in \cite[\S1]{Schutz2022}.)

\subsection{The count \texorpdfstring{$\mu$}{mu} and the even theory}

\begin{definition}\label{def:Nd}
For $m,n\ge2$ and $q\equiv m+n\pmod 2$ let
\[
\mu(m,n;q)\ :=\ \#\bigl\{(i_1,i_2)\in(2\Z)^{2}\ :\ 2\le i_1\le m-1,\ \ 2\le i_2\le n-1,\ \
i_1+i_2=d\bigr\}.
\]
\end{definition}

Thus $\mu(2,n;q)=0$, and for $m,n\ge3$ one has $\mu(m,n;q)\ge1$ if and only if $d$ is even and
$4\le d\le 2\lfloor\tfrac{m-1}{2}\rfloor+2\lfloor\tfrac{n-1}{2}\rfloor$, where
$2\lfloor\tfrac{m-1}{2}\rfloor$ and $2\lfloor\tfrac{n-1}{2}\rfloor$ are the largest even
integers that $i_1$ and $i_2$ can be. This count computes the even theory.

\begin{proposition}\label{prop:even}
Let $m,n\ge2$ and let $q\equiv m+n\pmod 2$. Then
\[
\Xe^{q}(L_{m,n})\ \simeq\ \bigl(\Sigma^{d-2}(\RP^{2}\wedge\RP^{2})\bigr)^{\wed \mu(m,n;q)}\wed W,
\]
where $W$ is a wedge of spheres and of suspensions of $\RP^{2}$. Consequently the second
Steenrod square $\Sq^{2}_{e}\colon\Kh^{\,i,q}(L_{m,n};\F_2)\to\Kh^{\,i+2,q}(L_{m,n};\F_2)$ of the even
theory has rank $\mu(m,n;q)$ for $i=d$ and is zero for $i\ne d$.
\end{proposition}

\begin{proof}
By Khovanov's computation \cite{Khovanov2000}, $\Kh(T(2,k);\Z)$ is obtained from the group
$\Kho(T(2,k);\Z)$ of \Cref{lem:torusknot} by replacing, for each even $i$ with
$2\le i\le k-1$, the two summands $\Z$ in bidegrees $(i,k+2i)$ and $(i+1,k+2i)$ by a single
summand $\Z/2$ in bidegree $(i+1,k+2i)$.
Since $T(2,k)$ is quasi-alternating, each $\Xe^{Q}(T(2,k))$ is a wedge of Moore spaces
\cite[Prop.~9.2]{LS14a}, hence determined by its cohomology: it is $\Sigma^{i-1}\RP^{2}$ if
$Q=k+2i$ with $i$ as above, and a wedge of at most two spheres otherwise.

By \eqref{eq:even-smash}, $\Xe^{q}(L_{m,n})$ is the wedge over $q_1+q_2=q$ of the spectra
$\Xe^{q_1}(T(2,m))\wedge\Xe^{q_2}(T(2,n))$. If one factor is a wedge of spheres, the summand
is a wedge of spheres and suspensions of $\RP^{2}$. Otherwise $q_1=m+2i_1$ and $q_2=n+2i_2$
with $i_1,i_2$ even, $2\le i_1\le m-1$ and $2\le i_2\le n-1$, and $q_1+q_2=q$ means
$i_1+i_2=d$. Conversely, every such pair $(i_1,i_2)$ gives the pair
$(q_1,q_2)=(m+2i_1,n+2i_2)$. So these summands correspond bijectively to the pairs counted by
$\mu(m,n;q)$, and each of them is
$\Sigma^{i_1-1}\RP^{2}\wedge\Sigma^{i_2-1}\RP^{2}\simeq\Sigma^{d-2}(\RP^{2}\wedge\RP^{2})$.
Taking $W$ to be the wedge of the remaining summands gives the stated decomposition.

The operation $\Sq^{2}$ is additive over wedge summands, and it vanishes on $W$, since the
cohomology of a sphere or a suspension of $\RP^{2}$ occupies at most two adjacent degrees.
The group $\widetilde H^{2}(\RP^{2}\wedge\RP^{2};\F_2)$ is spanned by $x\otimes x$, where
$x\in\widetilde H^{1}(\RP^{2};\F_2)$ is the generator, and $\Sq^{2}(x\otimes x)\neq0$
(\Cref{ss:chang}). Hence each summand
$\Sigma^{d-2}(\RP^{2}\wedge\RP^{2})$, whose mod $2$ cohomology lies in degrees $d$, $d+1$ and
$d+2$, contributes rank one to $\Sq^{2}_{e}\colon\Kh^{\,i,q}\to\Kh^{\,i+2,q}$ for $i=d$ and nothing
for $i\ne d$.
\end{proof}

For $L_{3,3}$ in grading $14$ both counts equal $1$, by \Cref{prop:even} and
\eqref{eq:schutz-example}. We conjecture that they agree in general.

\begin{conjecture}\label{conj:Nd}
For all $m,n\ge2$ and all $q\equiv m+n\pmod 2$,
\[
\nu(m,n;q)\ =\ \mu(m,n;q).
\]
Equivalently, by \eqref{eq:nu-Sq21} and \Cref{prop:even}, the operations $\Sq^{2}_{1}$ and
$\Sq^{2}_{e}$ have the same rank on
$\Kho^{\,i,q}(L_{m,n};\F_2)\cong\Kh^{\,i,q}(L_{m,n};\F_2)$ for $i=d$, where the isomorphism
is \cite[Prop.~1.7]{ORS}.
\end{conjecture}

\begin{remark}[Evidence]\label{rem:Nd-evidence}
Put $E_k=\{i\in2\Z:2\le i\le k-1\}$. The pairs counted by $\mu$ lie in
$E_m\times E_n$. They inject into the pairs counted by $P(d)$, and translating
each coordinate by $1$ gives an injection into those counted by $P(d+2)$.
Thus $\mu\le\min\{P(d),P(d+2)\}$, as required by \Cref{cor:room}.
Along a step in the second index,
\[
\mu(m,n;q)-\mu(m,n-1;q-1)
=[\,n\text{ odd and }d-n+1\in E_m\,].
\]
This is $0$ or $1$; if it is $1$, then $d-n+2\in S_m$, so the corresponding
source has $\gamma=1$. Moreover every counted coordinate is at most $d-2$,
which shows directly that $\mu$ is unchanged by the reduction
$(m,n)\mapsto(\rho_d(m),\rho_d(n))$. Its support lies in $4\le d\le m+n-2$.
These facts give the same vanishing, monotonicity, and stabilization constraints
as those proved for $\nu$.

\Cref{tab:kj} lists every grading with $\nu(m,n;q)\ne0$ for its fifteen pairs, and its entries,
including their multiplicities, are exactly the non-zero values of $\mu$. So $\nu=\mu$ in
every grading of these pairs. This is computational evidence, not a proof of the conjecture.
\end{remark}

The conjecture predicts
\[
\nu_d=\begin{cases}d/2-1,&d\ge4\text{ even},\\0,&d\text{ odd},\end{cases}
\]
by counting pairs of positive even integers adding to $d$. At the largest
possible window index $d=m+n-2$, it predicts $\nu=1$ precisely when both $m$ and $n$
are odd: the only possible pair is $(m-1,n-1)$.

\subsection{Open questions}

\begin{question}\label{q:odd-d}
Is $\nu(m,n;q)=0$ whenever $d$ is odd?
\end{question}

\Cref{conj:Nd} predicts a positive answer, since $\mu(m,n;q)=0$ whenever $d$ is odd, and \Cref{tab:kj}
contains no grading with $d$ odd, although for instance $L_{3,4}$ has room at $d=5$:
$\min\{P(5),P(7)\}=1$.

\begin{question}\label{q:refinements}
For all $m$ and $n$, do $\Sq^{2}_{0}$, $\Sq^{2}_{1}$ and the even square $\Sq^{2}_{e}$ coincide on
$\Kho(L_{m,n};\F_2)\cong\Kh(L_{m,n};\F_2)$?
\end{question}

By \Cref{prop:even}, an affirmative answer implies \Cref{conj:Nd}. By
\Cref{cor:classification} and \Cref{ss:oddspectrum} it also implies
$\X_{1}^{q}(L_{m,n})\simeq\X_{3}^{q}(L_{m,n})$ for all $q$, which is equivalent to
$\rk\Sq^{2}_{0}=\rk\Sq^{2}_{1}$ in every grading. For general links
$\X_{\ell}(L)$ is not determined by the parity of $\ell$ \cite[Thm.~1.1]{Rajapakse}, which answers
\cite[\S5.7, (q-7)]{SSS};
the two odd squares differ already on knots: among knots with at most eight crossings,
$\Sq^{2}_{0}$ is non-trivial only on $T(3,4)$ and $\Sq^{2}_{1}$ only on its mirror
\cite[Thm.~1.2]{Schutz2022}. For $L_{3,3}$ in grading $14$ both are isomorphisms
\cite[\S5.2]{Schutz2022}.

\begin{question}\label{q:cartan}
Let $L_1$ and $L_2$ be quasi-alternating, and let $\Sq^{1}_{e}$ denote the Bockstein of even
Khovanov homology. Is
\[
\Sq^{2}_{\varepsilon}(x\otimes y)\ =\ \Sq^{1}_{e}x\otimes\Sq^{1}_{e}y\qquad(\varepsilon=0,1)
\]
for all $x\in\Kh(L_1;\F_2)$ and $y\in\Kh(L_2;\F_2)$, under
$\Kho(L_1\sqcup L_2;\F_2)\cong\Kh(L_1;\F_2)\otimes\Kh(L_2;\F_2)$?
\end{question}

The Bockstein of odd Khovanov homology vanishes on $L_1$ and $L_2$, since $\Kho(L_i;\Z)$ is
free by \Cref{prop:orqa,prop:orsplit}, so the right-hand side is not seen by the odd theory
of the factors. Since each $\Xe(L_i)$ is a wedge of Moore spectra \cite[Prop.~9.2]{LS14a},
the Cartan formula and \eqref{eq:even-smash} identify the right-hand side with the even square
$\Sq^{2}_{e}$ of $L_1\sqcup L_2$. For $L_{m,n}$ the question is therefore \Cref{q:refinements}.
By \Cref{cor:classification}, which applies
to $L_1\sqcup L_2$ (\Cref{ss:future}), an affirmative answer would determine
$\Xo^{q}(L_1\sqcup L_2)$ by $\Kho(L_1\sqcup L_2;\Z)$ and the Bocksteins of $\Kh(L_1;\Z)$ and
$\Kh(L_2;\Z)$.

\begin{question}\label{q:hopf}
Is $\nu(2,n;q)=0$ for all $n\ge2$ and all $q$, that is, is every $\Xo^{q}(L_{2,n})$ a wedge
of spheres?
\end{question}

For $m=2$ the skein reduction has only two possible targets: if $0\le d\le n$ and $d\ne1$, then by
\Cref{thm:reduction}
\begin{equation}\label{eq:hopfred}
\nu(2,n;n+2+2d)\ =\
\begin{cases}\nu(2,d;3d+2), & n\in\{d,d+1\},\\ \nu(2,d+2;3d+4), & n\ge d+2,\end{cases}
\end{equation}
since here $M(q)=2$, while $N(q)=d$ if $n\in\{d,d+1\}$ and $N(q)=d+2$ if $n\ge d+2$
(substitute $q=n+2+2d$ in the definitions of $M(q)$ and $N(q)$). Hence the question concerns
only the values $\nu(2,d;3d+2)$ and $\nu(2,d+2;3d+4)$. \Cref{conj:Nd} predicts a positive
answer, as $\mu(2,n;q)=0$.

By \eqref{eq:hopfred}, for fixed $d$ the value $\nu(2,n;n+2+2d)$ is constant for $n\ge d+2$.
For $d=0$ this covers every $n\ge2$: $\nu(2,n;n+2)=\nu(2,2;4)$. For $d=2$, $\nu(2,n;n+6)$ equals
$\nu(2,2;8)$ for $n\in\{2,3\}$ and $\nu(2,4;10)$ for $n\ge4$. So the gradings $q=4$ and $q=8$ of
$L_{2,2}=\mathrm{Hopf}\sqcup\mathrm{Hopf}$ settle $d=0$ and the first two values for $d=2$. As $m+n=4$, these are the only
gradings of $L_{2,2}$ not covered by \Cref{thm:vanishing}, and in both
$\min\{P(d),P(d+2)\}=1$: \Cref{cor:room} leaves room for one summand $\Sigma^{d-2}\CP^{2}$. By \eqref{eq:deloop} and \Cref{lem:factorspectra},
$\Xo^{q}(\mathrm{Hopf}\sqcup U)\simeq\Xo^{q-1}(\mathrm{Hopf})\wed\Xo^{q+1}(\mathrm{Hopf})$, where
$\Xo^{q}(\mathrm{Hopf})$ is $\Sm^{0}$ for $q\in\{0,2\}$, $\Sm^{2}$ for $q\in\{4,6\}$ and
contractible otherwise. Since $L_{2,1}=\mathrm{Hopf}\sqcup T(2,1)$ and $T(2,1)$ is an unknot,
\eqref{eq:skeinLmn} with $m=n=2$ reads, in these two gradings,
\[
\begin{aligned}
\Sigma\Xo^{-1}(\mathrm{Hopf}\sqcup U)\simeq\Sm^{1}\ &\xrightarrow{\ \delta\ }\
\Xo^{3}(\mathrm{Hopf}\sqcup U)\simeq\Sm^{0}\wed\Sm^{2}\ \longrightarrow\ \Xo^{4}(L_{2,2}),\\
\Sigma\Xo^{3}(\mathrm{Hopf}\sqcup U)\simeq\Sm^{1}\wed\Sm^{3}\ &\xrightarrow{\ \delta\ }\
\Xo^{7}(\mathrm{Hopf}\sqcup U)\simeq\Sm^{2}\ \longrightarrow\ \Xo^{8}(L_{2,2}).
\end{aligned}
\]
By \Cref{prop:dichotomy}, $\Xo^{4}(L_{2,2})$ is $\Sm^{0}\wed(\Sm^{2})^{\wed2}$ or
$\Sigma^{-2}\CP^{2}\wed\Sm^{2}$, and $\Xo^{8}(L_{2,2})$ is $(\Sm^{2})^{\wed2}\wed\Sm^{4}$ or
$\CP^{2}\wed\Sm^{2}$. KnotJob gives $\Sq^{2}_{1}=0$ on $\Kho(L_{2,2};\F_2)$, so the first
alternative holds in both gradings: $\nu(2,2;4)=\nu(2,2;8)=0$, and $\Xo(L_{2,2})$ is a wedge of
spheres.
Hence $\nu(2,n;n+2)=0$ for all $n\ge2$, and $\nu(2,3;9)=0$.
KnotJob also gives $\Sq^{2}_{1}=0$ on $\Kho(L_{2,k};\F_2)$ for $3\le k\le10$, so
$\nu(2,n;q)=0$ for all $q$ whenever $n\le10$. In particular $\nu(2,4;10)=0$, so
$\nu(2,n;n+6)=0$ for all $n\ge2$. More generally, unless $0\le d\le n$ and
$d\ne1$, $\nu(2,n;q)=0$ by \Cref{thm:vanishing} or because $\Xo^{q}(L_{2,n})$ is contractible.
Otherwise, by \eqref{eq:hopfred}, $\nu(2,n;q)$ is a value of $\nu$ for $L_{2,d}$ if
$n\in\{d,d+1\}$, and for $L_{2,d+2}$ if $n\ge d+2$. Both links are covered by the computation
when $d+2\le10$, so $\nu(2,n;q)=0$ for all $n\ge2$ whenever $d\le8$. The case
$n\ge d+2$ for $d=9,10$, and every case with $d\ge11$, remain open, since they require
$L_{2,k}$ with $k\ge11$.

\subsection{Future directions}\label{ss:future}

\emph{Other thin factors.} Let $L_1,L_2$ be quasi-alternating. By
\Cref{prop:orqa,prop:orsplit} and \Cref{lem:kunneth}(ii), $\Kho(L_1\sqcup L_2;\Z)$ is free
and supported on three consecutive $\delta$-diagonals, so \Cref{cor:classification} applies
verbatim: $\Xo^{q}(L_1\sqcup L_2)$ is determined by its homology and the rank of
$\Sq^{2}_{1}$. The skein reduction uses more: resolving a crossing of the standard diagram
of $T(2,\pm n)$ produces $T(2,\pm(n-1))$ and an unknot, so that the term $A$ of the skein
sequence is a wedge of spheres by delooping (\Cref{lem:source} and \Cref{sec:mirror}), and the
interior grading of \Cref{thm:vanishing} comes from the gap
$1\notin S_k$. At a quasi-alternating crossing of a general factor the source of
\eqref{eq:skein} is again a split union of quasi-alternating links, and one needs a version
of \Cref{prop:dichotomy} for sources which are not wedges of spheres.

\emph{Mirrors.} For $\overline{L}_{m,n}=T(2,-m)\sqcup T(2,-n)$ the skein sequence runs in
the opposite direction; its Spanier--Whitehead dual has the form treated in
\Cref{ss:dichotomy}, and \Cref{thm:mirror} proves Theorems~\ref{thm:vanishing},
\ref{thm:reduction} and~\ref{thm:stabilization} for the mirror family. The relation between
the mirror invariant $\overline{\nu}$ and $\nu$ remains open (\Cref{rem:mirror-nu}).

\emph{More components.} For $r\ge3$ factors the odd Khovanov homology of
$T(2,k_1)\sqcup\dots\sqcup T(2,k_r)$ is still free, by \Cref{lem:kunneth}(ii) and induction,
and in each quantum grading it is concentrated in $r+1$ consecutive degrees. Besides $\eta$,
the element $\eta^{2}\in\pi^{s}_{2}$ may then attach a cell three degrees above another.
Since the integral cohomology is free, $\Sq^{1}$ vanishes in every degree, and hence so does
$\Sq^{3}=\Sq^{1}\Sq^{2}$ (by the Adem relation); on
the cofiber of $\eta^{2}$, whose cohomology sits in degrees $j$ and $j+3$, $\Sq^{1}$ and
$\Sq^{2}$ vanish because the cohomology in degrees $j+1$ and $j+2$ is zero, and $\Sq^{3}$
vanishes as above. So no primary operation detects $\eta^{2}$: \Cref{cor:classification} has
no analogue in terms of primary operations, and $\eta^{2}$ has to be detected by a secondary
one. In the
even theory the split union of three trefoils already has $\Sq^{3}\neq0$ \cite[\S2.3]{LOS}.

\appendix
\section{The mirror family}\label{sec:mirror}
Let $\overline{L}_{m,n}:=T(2,-m)\sqcup T(2,-n)$, $m,n\ge2$, be the mirror of $L_{m,n}$, with
the mirrored standard diagram; all its crossings are negative. For $q\equiv m+n\pmod2$ put
$d=(q-m-n)/2$ as before. The link $T(2,-k)$ is quasi-alternating with signature $1-k$ in the
convention of \Cref{prop:orqa}, and its Jones polynomial is $J_{T(2,k)}(q^{-1})$, so the proof
of \Cref{lem:torusknot} gives $\Kho^{\,i,q}(T(2,-k);\Z)\cong\Kho^{\,-i,-q}(T(2,k);\Z)$. With
\Cref{lem:kunneth}(ii), the proof of \Cref{lem:window} then shows that
$\Kho^{\,*,-q}(\overline{L}_{m,n};\Z)$ is free and concentrated in the degrees $-d-2$, $-d-1$,
$-d$, of ranks $P(d+2)$, $2P(d+1)$, $P(d)$. By \Cref{cor:classification} there is a unique
integer $\overline{\nu}=\overline{\nu}(m,n;q)\ge0$ such that
\begin{equation}\label{eq:mirror-classification}
\Xo^{-q}(\overline{L}_{m,n})\simeq(\Sigma^{-d-4}\CP^{2})^{\wed\overline{\nu}}
\wed(\Sm^{-d-2})^{\wed P(d+2)-\overline{\nu}}\wed(\Sm^{-d-1})^{\wed2P(d+1)}
\wed(\Sm^{-d})^{\wed P(d)-\overline{\nu}},
\end{equation}
and $\overline{\nu}$ is the rank of $\Sq^{2}_{1}$ from $\Kho^{\,-d-2,-q}(\overline{L}_{m,n};\F_2)$
to $\Kho^{\,-d,-q}(\overline{L}_{m,n};\F_2)$.

\begin{theorem}\label{thm:mirror}
Let $m,n\ge2$ and $q\equiv m+n\pmod2$, and let $M(q)$, $N(q)$ and $q^{*}$ be as in
\Cref{thm:reduction}.
\begin{enumerate}
\item[(i)] $\overline{\nu}(m,n;q)=0$ if $d\in\{-2,-1,1,m+n-1,m+n\}$.
\item[(ii)] If $0\le d\le m+n-2$ and $d\ne1$, then
$\Xo^{-q}(\overline{L}_{m,n})\simeq\Xo^{-q^{*}}(\overline{L}_{M(q),N(q)})\wed W$ with $W$ a wedge
of spheres; in particular $\overline{\nu}(m,n;q)=\overline{\nu}(M(q),N(q);q^{*})$.
\item[(iii)] If $d\ge3$ and $m,n\ge d+2$, then $\overline{\nu}(m,n;q)=\overline{\nu}(d+2,d+2;4d+4)$.
\end{enumerate}
\end{theorem}

Thus Theorems~\ref{thm:vanishing}, \ref{thm:reduction} and~\ref{thm:stabilization}, and with
them \Cref{cor:intro-irreducible}, hold for the mirror family with $\overline{\nu}$ in place of
$\nu$.

\begin{proof}
Part (i) follows from \Cref{cor:classification}, which gives
$\overline{\nu}\le\min\{P(d+2),P(d)\}$, and \Cref{lem:Pvanishing}, as in the proof of
\Cref{thm:vanishing}. For (ii) and (iii) we follow \Cref{sec:skein,sec:reduction}; only
\Cref{prop:dichotomy} needs a new argument.

Resolve a crossing of the factor $T(2,-n)$. Since it is negative, its $1$-resolution is the
oriented one (\Cref{rem:skeinneg}): $D_{1}$ is the standard diagram of $\overline{L}_{m,n-1}$,
and $D_{0}$ is $T(2,-m)$ split with an unknot diagram whose $n-1$ crossings are positive, so
$c'=n_{-}(D_{0})-n_{-}(D)=-n$. \Cref{rem:skeinneg} and \Cref{lem:deloop} give a cofibration
sequence
\begin{equation}\label{eq:mirror-skein}
T\ \xrightarrow{\ \delta'\ }\ A\ \longrightarrow\ \Xo^{-q}(\overline{L}_{m,n}),
\end{equation}
where
\[
T:=\Sigma^{-1}\Xo^{-(q-1)}(\overline{L}_{m,n-1}),\qquad
A:=\Sigma^{-n}\Bigl(\Xo^{-q+3n-2}\bigl(T(2,-m)\bigr)\vee\Xo^{-q+3n}\bigl(T(2,-m)\bigr)\Bigr).
\]
Since $(q-1)-m-(n-1)=q-m-n$, the same $d$ belongs to $\overline{L}_{m,n-1}$ in grading
$-(q-1)$, so $T$ has free homology in the degrees $-d-3$, $-d-2$, $-d-1$. By the proof of \Cref{lem:factorspectra},
$\Xo^{Q}(T(2,-m))$ is a wedge of spheres, with one sphere $\Sm^{-i}$ for each $i\in S_m$ with
$-Q\in\{m+2i-2,m+2i\}$; the count in the proof of \Cref{lem:source} then gives
\[
A\simeq(\Sm^{-d-2})^{\wed\gamma}\wed(\Sm^{-d-1})^{\wed\beta}\wed(\Sm^{-d})^{\wed\alpha}
\]
with $\alpha$, $\beta$, $\gamma$ as in \Cref{lem:source}. The map $\delta'$ vanishes on integral
cohomology. Indeed, by the K\"unneth argument of \Cref{lem:deltastar} it suffices to treat
$T(2,-n)$ alone, for which $D_{1}=T(2,1-n)$, $D_{0}$ is an unknot and $c'=-n$; then the source
of $\delta'^{*}$ is $\Kho^{\,0,\pm1}(U;\Z)$. For the sign $-1$ the target
$\Kho^{\,1-n,1-3n}(T(2,1-n);\Z)$ vanishes. For the sign $+1$ and $n\ge3$ the target
$\Kho^{\,1-n,3-3n}(T(2,1-n);\Z)\cong\Z$ maps onto $\Kho^{\,1-n,2-3n}(T(2,-n);\Z)\cong\Z$,
because the next group $\Kho^{\,1,1}(U;\Z)$ vanishes; so this map is injective and
$\delta'^{*}=0$ by exactness. For $n=2$ the target $\Kho^{\,-1,-3}(U;\Z)$ vanishes.

Now pass to Spanier--Whitehead duals, writing $Y^{\vee}$ for the dual of a finite spectrum $Y$.
Duality reverses cofibration sequences, so \eqref{eq:mirror-skein} gives a cofibration sequence
\[
\Xo^{-q}(\overline{L}_{m,n})^{\vee}\ \longrightarrow\ A^{\vee}\
\xrightarrow{\ \delta'^{\vee}\ }\ T^{\vee},
\]
and $\Xo^{-q}(\overline{L}_{m,n})^{\vee}\simeq\Sigma^{-1}\Cone(\delta'^{\vee})$. Put $e:=d+1$.
Then
\[
A^{\vee}\simeq(\Sm^{e-1})^{\wed\alpha}\wed(\Sm^{e})^{\wed\beta}\wed(\Sm^{e+1})^{\wed\gamma},
\]
the spectrum $T^{\vee}$ has free homology in the degrees $e$, $e+1$, $e+2$, and $\delta'^{\vee}$
vanishes on integral cohomology, since $H^{*}(Y^{\vee};\Z)\cong H_{-*}(Y;\Z)$ naturally and
$\delta'$ vanishes on integral homology by the universal coefficient theorem, all groups
being free. This is the situation of \Cref{lem:source} and \Cref{prop:dichotomy}, with $d$
replaced by $e$, and the proof of \Cref{prop:dichotomy} uses nothing else. Hence either
$\delta'^{\vee}\simeq0$, or $\gamma=1$ and $\Cone(\delta'^{\vee})$ has exactly one more summand
$\Sigma^{e-2}\CP^{2}$ than $T^{\vee}$.

Duality preserves wedges, sends spheres to spheres, and sends $\Sigma^{j}\CP^{2}$ to
$\Sigma^{-j-6}\CP^{2}$, because the dual of $\eta$ is again the non-zero element of
$\pi^{s}_{1}$. So $\Xo^{-q}(\overline{L}_{m,n})$ has as many summands $\Sigma^{j}\CP^{2}$ as
$\Cone(\delta'^{\vee})$, and $T$ as many as $T^{\vee}$. We obtain the analogues of
\Cref{prop:dichotomy} and \Cref{cor:deltanull}: for $n\ge3$, $\overline{\nu}(m,n;q)$ equals
$\overline{\nu}(m,n-1;q-1)$ or $\overline{\nu}(m,n-1;q-1)+1$, and if $\gamma=0$, then
$\delta'\simeq0$ and $\Xo^{-q}(\overline{L}_{m,n})\simeq\Xo^{-(q-1)}(\overline{L}_{m,n-1})\wed A$.
The proof of \Cref{thm:reduction} uses only \Cref{cor:deltanull}, so it gives (ii), and (iii)
is the case $m,n\ge d+2$ of (ii), as in the proof of \Cref{thm:stabilization}.
\end{proof}

\begin{remark}\label{rem:mirror-nu}
\Cref{thm:mirror} does not identify $\overline{\nu}$ with $\nu$. Sarkar, Scaduto and Stoffregen
ask whether $\Xo^{q}(\overline{L})\simeq\Xo^{-q}(L)^{\vee}$ for every link $L$, as in the even
theory \cite[\S5.7, (q-5)]{SSS}, \cite[Thm.~11]{LLS}; this would give $\overline{\nu}=\nu$, but it
fails for $T(3,4)$ \cite[Thm.~1.2]{Rajapakse}. Sch\"utz's computations suggest instead that
$\Sq^{2}_{\varepsilon}$ on $\overline{L}$ is the transpose of $\Sq^{2}_{1+\varepsilon}$ on $L$
\cite[Conj.~5.1]{Schutz2022}. This would give
$\overline{\nu}(m,n;q)=\rk_{\F_2}\bigl(\Sq^{2}_{0}\colon\Kho^{\,d,q}(L_{m,n};\F_2)\to
\Kho^{\,d+2,q}(L_{m,n};\F_2)\bigr)$, and hence $\overline{\nu}=\nu$ if \Cref{q:refinements} has an
affirmative answer.
\end{remark}

\section*{Tool use}
The ranks of $\Sq^{2}_{1}$ used in \Cref{sec:examples,sec:questions}, including those in
\Cref{tab:kj}, were computed with Sch\"utz's program KnotJob \cite{KnotJob}.

The author used the generative AI assistant Claude (Anthropic) in preparing this paper: to edit
and typeset the text in \LaTeX, to draw the figures in TikZ, to search for and check references,
and to check parts of the exposition and of the proofs. The mathematics is
the author's own: the main ideas, the results and their proofs are due to the author, who takes
full responsibility for the contents of the paper.
\bibliographystyle{amsalpha}
\bibliography{Odd_Skein_Reduction}

\end{document}